\documentclass[dvipsnames,reqno, a4paper, 11pt]{amsart}
\usepackage{young,pstricks,array,multido}
\usepackage{fullpage}
\usepackage[utf8x]{inputenc}
\usepackage[all]{xy}
\usepackage{color}
\usepackage{amssymb, amsmath}
\usepackage{rotating}
\usepackage{fourier}

\newcommand{\matdd}[4]{
\left (
\begin{array}{cc}
#1 & #2 \\
#3 & #4
\end{array}
\right ) }

\newcommand{\tabyoung}[3]{
\begin{center}
\begin{tabular}{ m{0mm} *{#1}{|>{\centering\arraybackslash}m{#2 mm}} | m{0mm} }
#3
\end{tabular}
\end{center}
}

\newcommand{\tabyoungnoncentre}[3]{
\begin{tabular}{ m{0mm} *{#1}{|>{\centering\arraybackslash}m{#2 mm}} | m{0mm} }
#3
\end{tabular}
}

\newcommand{\bd}[1]{\rule{0pt}{#1 mm}}

\newcommand{\pd}{\mathrm{pd}}
\newcommand{\depth}{\mathrm{depth}}
\newcommand{\acm}{arithmetically Cohen-Macaulay }

\newcommand{\im}{\mathrm{Im}}
\newcommand{\scal}[1]{\langle #1 \rangle}
\newcommand{\comment}[1]{}

\newcommand{\cC}{{\mathcal C}}
\newcommand{\wh}[1]{\widehat{#1}}
\newcommand{\tr}{{}^t}
\newcommand{\rank}{\mathrm{rank}}
\newcommand{\codim}{\mathrm{codim}}
\newcommand{\Pic}{\mathrm{Pic}}

\def\P{\mathbb{P}} \def\Q{\mathbb{Q}} \def\C{\mathbb{C}}
\def\Z{\mathbb{Z}}
\def\OP{{\mathcal{O}_{\P^2}}}
\def\Oc{{\mathcal{O}}}
\def\Oab{\mathcal{O}_{ab}}
\def\F{\mathcal{F}}
\def\G{{\textrm{G}_{ab}}}

\def\Om{{\Omega_{ab}}}
\def\OY{{\Omega_Y}}
\def\cL{{\mathcal{L}}}

\newtheorem{theorem}{Theorem}
\newtheorem{proposition}{Proposition}[section]
\newtheorem{lemma}[proposition]{Lemma}

\newtheorem{example}[proposition]{Example}
\newtheorem{notation}[proposition]{Notation}
\newtheorem{fact}[proposition]{Fact}

\theoremstyle{definition}
\newtheorem{definition}[proposition]{Definition}

\theoremstyle{remark}
\newtheorem{remark}[proposition]{Remark}

\theoremstyle{plain} 
\newtheorem*{theoA}{Theorem A}
\newtheorem*{theoB}{Theorem B}

\newtheorem*{coroC}{Corollary C}
\newtheorem*{theoD}{Theorem D}
\newtheorem*{theorem*}{Theorem }

\title[Stability of restrictions of cotangent bundles]{Stability of restrictions of 
the cotangent bundle of irreducible Hermitian symmetric spaces of compact type}

\author[I. Biswas]{Indranil Biswas}

\email{indranil@math.tifr.res.in}

\address{School of Mathematics, TIFR, Homi Bhabha Road, Bombay 400005, India}

\author[P.-E. Chaput]{Pierre-Emmanuel Chaput}

\email{pierre-emmanuel.chaput@univ-lorraine.fr}

\address{Institut Élie Cartan de Lorraine,
Université de Lorraine, 54506 Vandoeuvre-lès-Nancy Cedex, France}

\author[Ch. Mourougane]{Christophe Mourougane}

\email{christophe.mourougane@univ-rennes1.fr}

\address{IRMAR,
Campus de Beaulieu, b\^at. 22, Universit\'e Rennes 1,
263 av. G\'en\'eral Leclerc, 35042 Rennes,
France}



\date{\today}

\begin{document}

\begin{abstract}
It is known that the cotangent bundle $\Omega_Y$ of an irreducible Hermitian symmetric space $Y$ of compact type is stable. 
We show that if $X \subset Y$ is a subvariety whose structure sheaf has a short split resolution 
and such that the restriction map ${\Pic}(Y) \to {\Pic}(X)$ is surjective, 
then, apart from a few exceptions, the restriction $\Omega_{Y|X}$ is stable.
We then address some cases where the Picard group increases by restriction. 
\end{abstract}

\maketitle

\setcounter{tocdepth}{1}
\tableofcontents

\section{Introduction}
Throughout this article, by (semi-)stability of a vector bundle we mean \emph{slope (semi-)stabi\-lity} with respect to some fixed polarization.
The stable vector bundles with zero characteristic classes on a smooth projective variety
are given by the irreducible unitary representations of the fundamental group of the variety.
But, in the study of moduli spaces, it is a difficult and interesting question to produce explicit examples of stable vector bundles 
on algebraic varieties with non-zero characteristic classes.
There are such vector bundles within the framework of homogeneous spaces:
for example, an irreducible homogeneous bundle on a homogeneous space is
stable with respect to any polarization~\cite{Um}, \cite{ramanan}, \cite{indranil}. 
Once one has these stable vector bundles, theorems of Mehta-Ramanathan~\cite{me-ra}, Flenner~\cite{Fl} and Langer~\cite{la} 
assert the stability of the restriction of these bundles to a general hypersurface of high enough degree with respect to the given polarization. 
As, Balaji and Koll\'ar~\cite[question 4]{balaji-kollar}, asked if, having chosen a very ample polarization,
a stable reflexive sheaf on a normal projective surface restricts to a stable bundle on a general curve
of degree $4$ or more, we in fact expect quite small optimal bounds on the degree.

In this article, we address the following general questions: 
\begin{itemize}
\item Produce examples of large rank stable vector bundles on Fano manifolds of small degree.
\item More specifically, what can be said about the stability of the restriction of an irreducible homogeneous bundle $E$, 
defined on a homogeneous space $Y$, to a subvariety~$X\,\subset\, Y$ with emphasis on
giving good bounds on the degree and a description of the required features of the \emph{general} subvarieties~?
\end{itemize}

We study these questions in the following setting.
Recall that a \emph{Hermitian symmetric space} is a Hermitian manifold in which every point is an isolated point of an isometric involution.
It is homogeneous under its isometry group. 
It is called \emph{irreducible} if furthermore it cannot be written as the non trivial product of two Hermitian symmetric spaces.
This property is equivalent to the fact that the isometry group is almost simple.
Compact examples of irreducible Hermitian symmetric spaces consist of the usual Grassmannians, 
the quadric hypersurfaces\footnote{To deal with all quadrics, we handle the case of the two dimensional quadric, although it is reducible.}, 
the Lagrangian Grassmannians parameterizing $n$-dimensional Lagrangian subspaces of $\C^{2n}$ equipped with a symplectic form, 
the spinor Grassmannian parameterizing one family of $n$-dimensional isotropic
subspaces of $\C^{2n}$ equipped with a non degenerate quadratic form, and two exceptional manifolds.

The 2-dimensional quadric has Picard number $2$ and we consider its anti-canonical polarization.
All other examples have Picard number one, so that we do not have to specify the choice of a polarization.
We will denote the ample generator of the Picard group of $Y$ by $\Oc_Y(1)$.
We will denote by $\deg Y$ the top degree self-intersection $\deg Y:=\Oc_Y(1)^{\dim Y}$ 
and by $c_1(Y)$ the index of the Fano manifold $Y$. We recall that
$c_1(Y)$ is defined by the equality $-K_Y\,=\,c_1(Y) \Oc_Y(1)$ as elements of the N\'eron-Severi group of $Y$. 
The degree $\deg_Y X$ of a subvariety $X$ of $Y$ is defined such that
$\Oc_Y(1)_{\mid X}^{\dim X}=\deg_Y X\deg Y$.
In particular, if $X$ is given as the complete intersection of divisors in $\prod_{i=1}^c|\Oc_Y(d_i)|$
then $\deg_Y X=\Pi_{i=1}^c d_i$.

We assume first that $Y$ is a compact irreducible Hermitian symmetric space.
Its cotangent bundle $\Omega_Y$ is an irreducible homogeneous, hence stable, vector bundle.
We choose for $X$ a locally factorial positive dimensional subvariety whose structure sheaf has a short split resolution
(see Definition \ref{defi:short-resolution})
and such that the restriction $\Pic (Y) \to \Pic (X)$ is surjective. 
 By Lefschetz theorem~\cite{lefschetz}, this holds whenever $X$ is a complete intersection of dimension $\dim X\geq 3$ (see also \cite[Example 3.1.25]{lazarsfeld})
or whenever $X$ is a very general complete intersection surface in $\P^n$ except if it is a degree $d\leq 3$ surface in $\P^3$ 
or the intersection of two quadric threefolds in $\P^4$ (see also \cite[Theorem 1]{kim}).

\begin{theoA}[{Theorem~\ref{theo:restriction_hermitian} and
Theorem~\ref{theo:restriction_exceptions}}]
Let $Y$ be a compact irreducible Hermitian symmetric space, and let $X$ be a
locally factorial positive dimensional subvariety whose structure sheaf has a short split resolution in $Y$.
Assume that the restriction homomorphism ${\Pic}(Y) \to {\Pic}(X)$ is surjective. 
If $Y$ is a projective space or a quadric, assume moreover that $X$ has no linear equation.
Then, the restriction of $\Omega_Y$ to $X$ is stable.
\end{theoA}

Using a \emph{relative} Harder-Narasimhan filtration, if $Y$ be a compact irreducible Hermitian symmetric space, 
and $X$ a general complete intersections, we infer the semi-stability of restriction of $\Omega_Y$ to $X$
without assuming that the Picard group of $X$ comes from that of $Y$ (see Theorem \ref{theo:general-restriction_divisor}).

We then deal with the case of small dimensions where the Picard group increases by restriction.
Recall that irreducible Hermitian symmetric spaces of dimension $2$ or $3$ are $\P^2$, $\P^3$ and $\Q^3$.
 We add the similar case of $\Q^2$ for completeness.

\begin{theoB}[Theorem~\ref{theo:general-restriction_divisor}, {Theorem~\ref{theo:restriction_divisor}}]
Let $Y$ be an irreducible compact Hermitian symmetric space of dimension $2$ or $3$.
Let $X \subset Y$ be a smooth divisor, and in the case of ambient dimension $3$, let $C$ be a complete intersection curve.
\begin{itemize}
\item Take $Y=\P^2$. If $\deg_Y X \geq 2$, then $\Omega_{Y|X}$ is semi-stable 
and if $\deg_Y X \geq 3$, then $\Omega_{Y|X}$ is stable.

\item Take $Y=\P^3$. If $\deg_Y X \geq 2$, then $\Omega_{Y|X}$ is stable . 
If $C$ is the complete intersection of two general non-linear hypersurfaces, then $\Omega_{Y|C}$ is semi-stable.

\item Take $Y=\Q^3$.
If $\deg_Y X=1$, then $\Omega_{Y|X}$ is semi-stable and 
if $\deg_Y X\geq 2$, then $\Omega_{Y|X}$ is stable.\\
If $C$ is the complete intersection of two general non-linear hypersurfaces, then $\Omega_{Y|C}$ is semi-stable.
\end{itemize}
\end{theoB}

Note in particular that we have a complete answer to the question of stability of the restriction to \emph{divisors}:
\begin{coroC}
 Let $Y$ be an irreducible compact Hermitian symmetric space. Let $X \subset Y$ be a smooth divisor, not of degree 1.
 Then $\Omega_{Y|X}$ is stable with the only exception of a smooth conic in $\P^2$.
\end{coroC}

Our arguments are very different in the two situations of theorem A and B.
In the case of no increase of Picard groups, we use a new vanishing theorem that may be of independent interest:
\begin{theoD}[Theorem~\ref{theo:general_inequalities}]
Let $Y$ be a compact irreducible Hermitian symmetric space, but not a projective space.
Let $l,p,q$ be integers, with $l>0$, $p>0$, such that $H^q(Y,\Omega^p_Y(l))\neq 0$. 
Then, $$l+q \geq p \frac{c_1(Y)}{\dim(Y)}.$$
If moreover $q>0$, then $$l+q \leq p.$$
\end{theoD}

The proof of this theorem relies on a combinatorial study of the Bott vanishing theorem~\cite{bott}. 
Note that by Akizuki-Kodaira-Nakano theorem \cite{an}, we have $p+q \leq \dim Y$
under the assumption $H^q(Y,\Omega^p_Y(l))\neq 0$ of the theorem. 
On the other hand, by Serre's asymptotic vanishing theorem, we also expect in the 
case $q>0$, an upper bound on the value of $l$. Finally, by the work of Snow 
\cite{snow,snow2}, we know that $q \leq p$. However, to the best of our knowledge, 
it is the first time that the idea that a bound on the \emph{sum} $l+q$ could hold 
appears, addressing the question whether a similar bound could hold for some more 
general varieties $Y$. It should be mentioned that a somewhat related vanishing 
theorem was proved in~\cite{faenzi}.

The connection between vanishing results and stability is as follows: by standard cohomological arguments, 
as used for example in ~\cite{PW},
a subbundle $\F \subset \Omega_{Y|X}$ of the restriction of $\Omega_Y$
contradicting stability yields, under the assumption of the existence of a short resolution of $\Oc_X$ by split vector bundles on $Y$, 
the non-vanishing of some cohomology group $H^q(Y,\Omega^p(l))$, where $q$ is related to the codimension of $X$
and $l$ to the degree of $\F$. Our vanishing theorem implies the desired stability inequality (see Section~\ref{section:hermitian}). 
Remarkably enough, the bounds in Theorem~D are exactly the bounds we need to establish the stability inequality.

In small dimensions, we use tools from projective geometry to make explicit the new line 
bundles that appear on the subvariety $X$. We found it noteworthy that the precise 
stability inequalities ultimately originate in some Bezout like theorems. Those cases 
with increase of the Picard group seem untractable either through vanishing theorems,
or through the technique of covering rational curves as in ~\cite{HW}. 
In a particular case, we need some arguments of representation theory (see Section~\ref{section:small dimensions}).

Our initial motivation for the present work is its implication in terms of height inequalities:
for example, starting with a manifold $W$ with ample canonical bundle $K_W$
(e.g. an hypersurface of large degree in a projective manifold $Y$ of Picard number one), 
if the restriction of its cotangent bundle $\Omega_W$ to a curve $C\,\subset\, X$
is semi-stable then, as $\Omega_C$ is a quotient of ${\Omega_W}_{\mid C}$, 
the slope inequality gives a bound on the height of $C$ with respect to the canonical polarization $K_W$
in terms of the genus of $C$. We hope to push further the present study of ${\Omega_Y}_{\mid C}$
to deal with the stability properties of ${\Omega_W}_{\mid C}$.

\subsection*{Acknowledgements}

We thank Fr\'ed\'eric Han and Laurent Gruson for their help in projective 
geometry. We thank Michel Brion for pointing out a mistake in a previous version of 
this paper. We thank Jie Liu for pointing out the improvement in the use of Langer's
technique for $\Q^3$.
The first-named author is supported by a J. C. Bose Fellowship. Part of 
this project was done in the TATA Institute (Bombay); we thank it for
hospitality.

\section{Vanishing theorems}
\label{vanishing}

Let $Y$ be a compact irreducible Hermitian symmetric space.
We embed $Y$ in some projective space $\P^N$ thanks to a homogeneous ample line bundle $L$.
For a sheaf $\F$ on $Y$ and an integer $\ell$, we will denote
the tensor product $\F \otimes L^{\otimes \ell}$ by $\F(\ell)$.

\subsection{The statement}\label{vanishing31}
The current section is devoted to prove the
\begin{theorem}
\label{theo:general_inequalities}
Let $Y$ be a compact irreducible Hermitian symmetric space, but not a projective space.
Let $l,p,q$ be integers, with $l\geq 0$ and $p> 0$, such that
$H^q(Y,\Omega^p_Y(l))\neq 0$.
\begin{enumerate}
\item Then, $$l+q \geq p \frac{c_1(Y)}{\dim(Y)}.$$
Furthermore, equality holds if and only if 
\begin{itemize}
\item $p=\dim(Y), q=0$ and $l=c_1(Y)$, 
\item or $Y$ is a quadric and $l=0$,
\item or $Y \simeq \Q^4,l=2,p=3,q=1$.
\end{itemize}
\item If moreover $q >0$, then
$$l+q \leq p.$$
\end{enumerate}
\end{theorem}

The proof of Theorem~\ref{theo:general_inequalities} is given in the next subsections. 
Surprisingly enough, the first item is very intricate,
and the proof of the vanishing theorem in this case entails involved combinatorial arguments.
The second item is much easier.

\subsection{The case of Grassmannians (type $A_n$)}

Fix positive integers $a,b \geq 2$.
Let $\G\,:=\,G(a,a+b)$ be the Grassmannian that parametrizes $a$-dimensional
linear subspaces of a fixed $(a+b)$-dimensional $\C$-vector space $V$.
It is the homogeneous space $G(a,a+b)\,=\,{\rm SU}(a+b)/[{\rm SU}(a+b)\cap
{\rm U}(a)\times {\rm U}(b)]$.
Let $\Oab(1)$ be the Pl\"ucker polarization on $\G$, which is also the positive
generator of $\Pic(\G)$.
The cotangent bundle of $\G$ is denoted by $\Om$.

In case $Y$ is a Grassmannian, by \cite{snow}, the non-vanishing of $H^q(Y,\Omega^p_Y(l))$
implies the existence of a partition of $p$, which is $l$-admissible with cohomological
degree $q$ in the following sense.
Recall first that given a non negative integer $p$, a \emph{partition of $p$} is a sequence of non-increasing natural integers
$(\lambda_1 \geq \lambda_2 \geq \cdots \geq \lambda_n)$ such that $\sum_{i=1}^n p_i=p$. It will be represented by its Young diagram:
each part is represented by $\lambda_i$ boxes on the $i$-th row (from top to bottom), and
these rows are left-justified. The dual partition $\lambda^\vee$ of a partition $\lambda$ is defined
by the fact that its $i$-th part $\lambda^\vee_i$ is equal to the number of boxes in the $i$-th column of the Young
diagram of $\lambda$.
The \emph{hook number} $h_\lambda(i,j)$ of a given box $(i,j)$ in a Young diagram is the number of boxes $i+j-1$
that build the hook based at the given box, as illustrated in the following example for the partition $(6,4,2,2)$
with dual partition $(4,4,2,2,1,1)$:
\tabyoung{6}{4}{
  \cline{2-7}
  \rule{0pt}{6mm} &&&&&&& \\
  \cline{2-7}
  \rule{0pt}{6mm} &6&*&*&* \\
  \cline{2-5}
  \rule{0pt}{6mm} &*&\\
  \cline{2-3}
  \rule{0pt}{6mm} &*&\\
  \cline{2-3}
}

\begin{definition}[\cite{snow}]\label{defi-admissible}
Let $\lambda$ be a partition and $l$ an integer. We say that $\lambda$ is \emph{$l$-admissible} if no hook number of
$\lambda$ is equal to $l$. The ($l$-){\it cohomological degree} of an $l$-admissible partition is the number of hook numbers
which are greater than $l$. The hook number of the partition $\lambda$ at the box $(i,j)$ will be denoted by
$h_\lambda(i,j)$.
\end{definition}

Hence, in the case of Grassmannians, the first part of Theorem~\ref{theo:general_inequalities} is equivalent to the following
combinatorial statement, which we now prove:
\begin{proposition}\label{prop-partition}
Assume that $a,b \,\geq\, 2$.
Let $l$ be a non-negative integer.
Let $\lambda$ be a $l$-admissible partition of $p$ which Young diagram fits in a rectangle $a\times b$, with
$l$-cohomological degree $q$.
Then, we have the inequality $$l+q\,\geq\, \frac pa + \frac pb\, ,$$ with equality
holding if and only if
either $\lambda=(0)$ and $l=q=0$, or $\lambda\,=\,(b^a)$, $l\,=\,a+b$ and $q\,=\,0$, or $\lambda \,=\, (2,1)$,
$l\,=\,2$ and $q\,=\,1$.
\end{proposition}

\begin{proof}

Given a partition $\lambda$, we denote by $a(\lambda)$ the first part of $\lambda$, by $b(\lambda)$ its length, by
$p(\lambda)$ the sum of its parts, and by $q(\lambda)$ its cohomological degree. We assume that $a=a(\lambda)$
and $b=b(\lambda)$.
If $q=0$, we have $\frac pa + \frac pb \leq b+a \leq l$ and the first inequality is an equality if and only if
$\lambda=(b^a)$. We are in the first case of the proposition.


We assume from now on that $q>0$.
Without loss of generality, we may also assume that $a \geq b$ and that the proposition is proved for partitions
$\lambda'$ such that $a(\lambda')<a$ or $b(\lambda')<b$.

Given an $l$-admissible partition $\lambda$, we denote by $\Delta(\lambda)$ the number
$$\Delta(\lambda):=l+q(\lambda)-\frac {p(\lambda)}{a(\lambda)} - \frac {p(\lambda)}{b(\lambda)}.$$
Our aim is to prove that $\Delta(\lambda) \geq 0$. It will be convenient to prove this inequality
repeatedly replacing $\lambda$ by a combinatorially simpler $l$-admissible partition $\lambda'$
such that $\Delta(\lambda') \leq \Delta(\lambda)$. The main steps of the proof are illustrated
in Picture 1.

Let us first assume that $h_\lambda(2,1)>l$. In this case, we consider the partition $\mu$ obtained
removing the first column: namely, we set $\mu_i=\max(\lambda_i-1,0)$. By induction, we know that
$\Delta(\mu) \geq 0$. Now, we have
$a(\mu)=a-1, b(\mu) \leq b, p(\mu)=p-b$ and $q(\mu) \leq q-2$. It follows:
$$
0 \leq \Delta(\mu) < l+(q-2)-\frac{p-b}{a}-\frac{p-b}{b}=l+q-\frac pa-\frac pb + \frac ba - 1 \leq \Delta(\lambda),
$$
where the last inequality follows from our assumption $a \geq b$.


From now on, assume that $h_\lambda(2,1)<l$.
Set $k=\lambda_{q+1}^\vee$. We consider the partition $\mu$ defined by (see Picture 1)
$$
\left \{
\begin{array}{rcl}
 \mu_1=\lambda_1 \\
 \mu_i = \lambda_2 & \mbox{ if } & 2 \leq i \leq k \\
 \mu_i = \lambda_k & \mbox{ if } & k+1 \leq i \leq b
\end{array}
\right .
$$
For all integers $i,j$, we have $h_\mu(i,j) \geq h_\lambda(i,j)$, and we have
$h_\mu(1,q+1)=h_\lambda(1,q+1)<l$ and $h_\mu(2,1)=h_\lambda(2,1)<l$. It follows that
$\mu$ is also $l$-admissible.
We have $a(\mu)=a, b(\mu)=b, p(\mu) \geq p, q(\mu)=q$, so $\Delta(\mu) \leq \Delta(\lambda)$,
so that it is enough to prove the Proposition for $\mu$.

We now consider two cases. The first case is when $\mu_2=q$. Then
the Young diagram of $\mu$ is described by the three integers
$a,b$ and $c$ with $a\geq b\geq 2$ and $a\geq c\geq 1$:
\tabyoung{12}{8}{
   \cline{2-13}
   \bd{9} &$a+b-1$  && $\cdots$ &&$a+b-c$ & $a-c$ & &&  &&  & $1$ &  \\
   \cline{2-13}
   \bd{9} & $b+c-2$ &&  && $b-1$ \\
   \cline{2-6}
   \bd{9} &  &&&&  \\
   \cline{2-6}
   \bd{9} &  &&  &&  \\
   \cline{2-6}
   \bd{9}  & $c$  &&  && 1\cr
   \cline{2-6}
}
Then, as all integers between $1$ and $a-c$ are hook numbers, and
similarly
for integers between $1$ and $b+c-2$, the integer $l$ has to fulfill the
following inequalities
$$ l\geq a-c+1,\ \ l\geq b+c-1.$$
Since the cohomological degree $q$ is positive, then, as all
numbers between $a+b-c$ and $a+b-1$ are hook numbers,
$$l\leq a+b-c-1.$$
There are exactly $c$ boxes with hook number bigger than $l$, all lying
in the first row: hence, $q=c$
and $$l+q\geq (a-c+1)+c\geq a+1.$$
Writing $p=a+(b-1)c=ab-(b-1)(a-c)$ we derive
$$\frac{p}{a}+\frac{p}{b}=a+1+(b-1)(\frac{c}{a}-\frac{a-c}{b}).$$
But
$$\frac{a-c}{b}-\frac{c}{a}=\frac{a(a-c)-bc}{ab}\geq
\frac{ac-bc}{ab}\geq 0,$$
where we have used $a-c\geq l-b+1\geq c$.
This proves $$\frac{p}{a}+\frac{p}{b}\leq l+q.$$
In case of equality, $(b-1)(ac-bc)=0$ so that $a=b$. Furthermore, all
previously used inequalities are equalities
$$l=a-c+1=b+c-1=a+b-c-1$$
leading to $c=1$ and $a=b=2$.


The second case is when $\mu_2>q$. Note that the inequalities $h_\mu(2,1)<l$
and $h_\mu(1,1)>l$
imply that $\mu_2<\mu_1$. Similarly, $\mu_{k+1} < \mu_2$.
If $h_\mu(1,q)>l+1$, we may consider the $l$-admissible partition $\nu$
obtained by setting $\nu_1=\mu_1-1$, $\nu_{k+1}=\mu_2$, and for
$i \not \in \{1,k+1\},\nu_i=\mu_i$. We have
$p(\nu) \geq p(\mu),a(\nu)=a-1,b(\nu)=b$ and $q(\nu)=q$,
so $\Delta(\nu)<\Delta(\mu)$ and the Proposition is proved
by induction in this case.
If $h_\mu(1,q)=l+1$, then $h_\mu(1,1)=l+q$, so $l+q=a+b-1$. Since
$p(\mu) \leq ab-b$, the inequality $\Delta(\mu) > 0$ is also proved in this
case.

\psset{unit=4mm}
\psset{linewidth=.1}

$$
\begin{array}{c}

\begin{pspicture}(25,10)
 \psframe*[linecolor=gray](6,10)(0,9)
 \psline(0,0)(0,10)(25,10)(25,9)(14,9)
 \psline(14,9)(14,8)
 \pscurve(14,8)(9,7.2)(7,5)
 \psline(7,5)(6,5)
 \pscurve(6,5)(3,3)(1,0)
 \psline(1,0)(0,0)
 \put(20,4){{\Huge $\lambda$}}
\end{pspicture}

\\

\begin{pspicture}(25,10)
 \psframe*[linecolor=gray](6,10)(0,9)
 \psline(0,10)(25,10)(25,9)(14,9)(14,5)(6,5)(6,0)(0,0)(0,10)
 \put(6.2,9.25){{\scriptsize $\bullet \longleftarrow h_\mu(1,q+1)=h_\lambda(1,q+1)$}}
 \put(0.2,8.25){{\scriptsize $\bullet \longleftarrow h_\mu(2,1)=h_\lambda(2,1)$}}
 \put(20,4){\Huge{$\mu$}}
\end{pspicture}

\\

\begin{pspicture}(25,10)
 \psframe*[linecolor=gray](6,10)(0,9)
 \psline(0,0)(0,10)(21,10)(21,9)(14,9)(14,1)(6,1)(6,0)(0,0)
 \put(6.2,9.25){{\scriptsize $\bullet \longleftarrow h_\nu(1,q+1)=h_\lambda(1,q+1)$}}
 \put(5.2,9.25){{\scriptsize $\bullet$}}
 \put(4.8,8.25){{ $\uparrow$}}
 \put(4.5,7.5){{\scriptsize $h_\nu(1,q)=l+1$}}
 \put(20,4){\Huge{$\nu$}}
\end{pspicture}

\\
Picture 1
\end{array}
$$
\end{proof}

We now prove the second part of Theorem~\ref{theo:general_inequalities}.

\begin{proposition}[Second part of Theorem~\ref{theo:general_inequalities} for Grassmannians]
\label{prop:An-2}
Let $l$ be a positive integer.
Let $\lambda$ be a $l$-admissible partition of $p$ 
with $l$-cohomological degree $q> 0$.
Then, we have the inequality $$l+q \leq p.$$ 
Moreover, if the equality $p=l + q$ holds, then $\lambda$ is a hook (i.e.,
its shape is $(a,1^{b-1})$).
\end{proposition}

\begin{proof}
Let $Y(\lambda)\,:=\,\{(i,j)\,\mid\, j \leq \lambda_i \}\,
\subset\, \mathbb{N}^2$ be the Young diagram of $\lambda$. For $x \,\in\, Y(\lambda)$, we
abbreviate the hook number $h_\lambda(x)$ of $\lambda$ at $x$
by $h(x)$. We have $p=\# Y(\lambda)$ and
$q=\# \{ x \in Y(\lambda)\ |\ h(x) > l \}$.
Since $q>0$, let $x \in Y(\lambda)$ such that $h(x)>l$. Moreover, we can assume that
$x$ is minimal for this property, namely that $h(y)<l$ if $y$ is south-east from $x$.
By definition of $h(x)$, there are $h(x)-1$ elements $z \in Y(\lambda)$ which are
either on the same row as $x$ on its right, or under $x$ in the same column.
For these elements, we have $h(z)<l$. This implies that
$p-q = \#\{y \in Y(\lambda)\ |\ h(y)<l \} \geq h(x)-1 \geq l$.

We now deal with the case of equality (that will not be used in the sequel).
If the equality $p=l + q$ occurs, with $q>0$ as above, then we first show that $x$ is on the first row.

\begin{center}
\begin{tabular}{cc}
\tabyoungnoncentre{5}{4}{
  \cline{2-6}
  \bd{6} & &&&& &  \\
  \cline{2-6}
  \bd{6} & &&&& $x''$  \\
  \cline{2-6}
  \bd{6} & & $x$ &&& $x'$  \\
  \cline{2-6}
  \bd{6} & & \\
  \cline{2-3}
  \bd{6} & & \\
  \cline{2-3}
  \bd{6} & & \\
  \cline{2-3}
  \bd{6} & \\
  \cline{2-2}
}
&
\tabyoungnoncentre{9}{4}{
\cline{2-10}
  \bd{6} & &&&&&&&& &  \\
  \cline{2-10}
  \bd{6} & &&&&&&&& $x''$  \\
  \cline{2-10}
  \bd{6} & & $x$ &&&& $x'$  \\
  \cline{2-7}
  \bd{6} & & \\
  \cline{2-3}
  \bd{6} & & \\
  \cline{2-3}
  \bd{6} & & \\
  \cline{2-3}
  \bd{6} & \\
  \cline{2-2}
}
\\
\ \\
Case $\lambda_{i-1} = \lambda_i$ & Case $\lambda_{i-1} > \lambda_i$
\end{tabular}
\end{center}

If $x$ is not on the first row, then the hook number of the box $x''$, very right on the row over that of $x$
not in the same row of $x$ neither on the same column, is $2$ or $1$. 
Hence this box contributes to  $\{ y \in Y(\lambda)\ |\ h(y) < l \}$ and therefore, $p-q>l$.
In the same way, we can show that $x$ is on the first column and that all the boxes with hook number 
smaller than $l$ are on the hook of $x$. Therefore $\lambda$ is a hook.
\end{proof}

\subsection{The case of quadrics (type $B_n$ or $D_n$)}
\label{sec:quadric}
Let $Y$ be a non singular quadric hypersurface of dimension $n$ 
with its natural polarization $\Oc_Y(1)=\Oc_{\P^{n+1}}(1)_{\mid Y}$.
It is the homogeneous space $$Y\,=\,{\rm SO}(n+2)/({\rm SO}(n)\times
{\rm SO}(2))\, .$$
{}From the adjunction formula, $c_1(Y)=n+2-2=\dim Y$.
Recall a theorem of Snow.
\begin{theorem*}[{\cite[page 174]{snow}}]
\label{theo:snow}
Let $Y$ be a non singular quadric hypersurface of dimension $n$.
If $H^q(Y,\Omega^p_Y(l))\neq~0$, then 
\begin{itemize}
  \item either $p=q$ and $l=0$,
  \item or $q=n-p$ and $l=-n+2p$,
  \item or $q=0$ and $l>p$,
  \item or $q=n$ and $l<-n+p$.
\end{itemize}
\end{theorem*}

As when $q=0$ and $0<p<n$, the inequality $l>p$ holds, 
the cotangent bundle $\Omega_Y$ of the quadric is stable as soon as the Picard group
of the quadric is the restriction of that of $\P^{n+1}$, for example when $n\geq 3$.
Theorem~\ref{theo:general_inequalities} follows for quadrics by checking the above cases.

\subsection{The case of Lagrangian Grassmannians (type $C_n$)}

In this case, $Y$ parametrizes $n$-dimensional Lagrangian
subspaces of $\C^{2n}$ equipped with the standard symplectic form.
It is the homogeneous space $Y\,=\,{\rm Sp}(2n, {\mathbb C})/{\rm U}(n)$.
By \cite{snow2}, the non-vanishing of $H^q(Y,\,\Omega^p_Y(l))$ amounts to the existence 
of an $l$-admissible $C_n$-sequence of weight $p$ and cohomological degree $q$, in the
following sense:
\begin{definition}
Fix $l,n\in\mathbb{N}$ with $l>0$.
A $n$-uple of integers $(x_i)_{1 \leq i \leq n}$ will be called an
\emph{$l$-admissible $C_n$-sequence} if
\begin{itemize}
\item $\forall\ 1\leq i\leq n$, $|x_i| = i$
\item $\forall\ i \leq j, x_i + x_j \not = 2l$.
\end{itemize}
Its weight is defined to be $$p\,:=\,\sum_{x_i>0} x_i$$ and its cohomological degree is
$q\, :=\, \# \left \{ (i,j)\,\mid\, i \leq j \mbox{ and } x_i + x_j > 2l \right \}$.
\end{definition} 

\begin{notation}
\label{nota:x+}
Given an integer $x$, we denote $x^+:=Max(0,x)$. Therefore, we have $p\,=\,\sum_i x_i^+$.
The set of all $(u,v)\,\in\, \mathbb{N}^2$ such that $u\leq v$ and $x_u+x_v>2l$
will be denoted by $Q$. The cardinality $\#\, Q$ will be denoted by $q$.
Moreover we adopt the following convention: if $\cC$ is a condition on $(u,v)$, then $Q(\cC)$ will denote
the subset of $Q$ consisting of pairs satisfying $\cC$. For example, given an integer $v_0$, the set
$Q(v=v_0)$ consists of all pairs $(u,v)$ in $Q$ such that $v=v_0$.
\end{notation}

Since this excludes the case of $Y$ being a projective space or a quadric, which occurs when $n \leq 2$,
the first part of Theorem~\ref{theo:general_inequalities} amounts to the following proposition in this case:
\begin{proposition}[First part of Theorem~\ref{theo:general_inequalities} for Lagrangian Grassmannians]
\label{prop:Cn-1}
Let $x=(x_i)$ be an $l$-admissible $C_n$-sequence of weight $p$ and cohomological degree $q$, with $n \geq 3$. Then
$$ l+q \geq \frac{2p}{n} \ , $$
with equality occurring if and only if $x_i=i,l=n+1$ and $q=0$, or $p=q=l=0$.
\end{proposition}

\begin{proof}
Let $t = \# \{ i \ | \ x_i > l \}$.
Snow classified the cases where $l=1$ \cite[Theorem 2.2]{snow2}.
In fact, the combinatorics are quite simple in this case since for a $2$-admissible sequence $(x_i)$
we have $x_1=-1$ and $x_i<0 \Rightarrow x_{2+i}<0$. Thus, such a sequence satisfies $x_{2i+1}=-(2i+1)$,
$x_{2i} = 2i$ for $i \leq t$, and $x_{2i} = -2i$ for $i>t$.
We then have $p=t(t+1)$ and $q=t^2$.
Since $t \leq \frac n2$, we have
$$ \frac{2p}{n} = \frac{2t(t+1)}{n} \leq t+1 \leq t^2+1 = q+l \ .$$
Moreover, if the equality holds, then $2t=n$ and $t=0$ or $t=1$, contradicting $n \geq 3$.
Therefore, the proposition is true in this case.

We now assume that $l \geq 2$. Given $i,j$, if $x_i > l$ and $x_j>l$, then
evidently $x_i + x_j > 2l$. Therefore,
$$q \,\geq\, \frac{t(t+1)}{2} \,\geq\, 2t-1\, .$$ On the other hand, we have
$p \leq \frac{l(l-1)}{2} + tn$. If $2p \geq (l+q)n$, then $$l(l-1) + 2 tn 
\,\geq\, (2t+l-1)n\, .$$ Since $l>1$, this implies $n \leq l$, and so $q=0$.

If $l=n$, then $x_n=-n$, so we have $2p \leq n(n-1)$, therefore,
$\frac{2p}{n} \leq n-1 < l$, and the proposition is true.

If $l>n$, since $2p \leq n(n+1)$, we get that $\frac{2p}{n} \leq n+1 \leq l$,
and if the equality holds then $l=n+1$ and $p=\frac{n(n+1)}{2}$.
\end{proof}

We now prove the second part of Theorem~\ref{theo:general_inequalities}:
\begin{proposition}[Second part of Theorem~\ref{theo:general_inequalities} for Lagrangian Grassmannians]
\label{prop:Cn-2}
Let $(x_i)$ be an $l$-admissible $C_n$-sequence of weight $p$ and cohomological degree $q$, with $n \geq 3$
and $q>0$. Then
$$ l+q \leq p \ . $$
Moreover, the equality $p=q+l$ holds if and only if 
$$x=(-1,-2,\ldots,-l,l+1,-(l+2),-(l+3),\ldots,-n)$$
with $p=l+1$ and $q=1$.
\end{proposition}
\begin{proof}
The proof is similar to that of Proposition \ref{prop:An-2}.
Let $j$ be the minimal integer such that there exists $i \leq j$ with $x_i+x_j > 2l$. 
We have $x_j=j \geq l+1$.
We want to bound $q=\#\, Q$. We observe that if $(u,j) \in Q$ with $j-l \leq u <j$, then $x_u>0$. 
Otherwise, $x_u=-u$ and $0<x_u+x_j=-u+j\leq l$, contradicting the
assumption that $x_u+x_j>2l$. 
Hence $1\leq x_u^+$, and therefore
$$\#\, Q(v=j,j-l \leq u <j) \,\leq\, \sum_{j-l \,\leq\, u <j} x_u^+\, .$$ 
Actually, a similar inequality holds with $j-l$ replaced by $j-2l$, 
but in the sequel we will use the inequality $j-l \geq 1$. 
In fact, we have $$\#\, Q(v=j,u<j-l) \,\leq\, j-l-1\, .$$ 
Finally, $\#\, Q(v>j) \leq \sum_{v>j} x_v^+$. 
Therefore, by minimality of $j$, we have the inequality:
$$
\begin{array}{rcl}
  q & = & \#\, Q(v=j=u) + \#\, Q(v=j,j-l\leq u < j) + \#\, Q(v=j,u<j-l) + \#\, Q(v>j) \\
  & \leq & 1 + \sum_{j-l \leq u < j} x_u^+ + (j-l-1) + \sum_{v>j} x_v^+ \\
  & \leq & \sum_{u<j} x_u^+ + x_j + \sum_{v>j} x_v^+ - l = p-l\ .
\end{array}
$$
This proves the inequality $l+q\leq p$.

\smallskip

We now deal with the case of equality. Assume that $q=p-l$. Then, 
asking for equalities in the previous estimates, we find that
for $j-l \leq u <j$, if $x_u>0$, then $x_u=1$, and 
for $u\,<\,j-l$, $x_u+j\,>\,2l$ by the first inequality and $x_u\,<\,0$ by the second. 
In particular, $j-l-1\leq 0$ and hence $j=l+1$ for otherwise
$x_{j-l-1}+j=-(j-l-1)+j=l+1>2l$.
For $v_0$ such that $j=l+1<v_0$, from the equality $Q(v=v_0)=x_{v_0}^+$, we infer that if $x_{v_0}>0$
then for all $u\leq v_0$, $x_u+x_{v_0}>2l$. In particular, $x_{v_0-1}>0$ and $x_{v_0-2}>0$.
By decreasing induction, we find that $x_l=l$, contradicting the $l$-admissibility.
Hence, for $j<v<2l$, we get $x_v<0$.
Finally, $x$ is of the form $(-1,-2,-3,\cdots,-l,l+1,-(l+2),-(l+3),\ldots,-n)$ or 
$(1,-2,-3,\cdots,-l,l+1,-(l+2),-(l+3),\ldots,-n)$.
In the second case, one has $p=l+2$ thus $q=2$ thus $x_1+x_{l+1} > 2l$ thus $l=1$. But then $x$ is not
$1$-admissible since $x_1=1$.
\end{proof}

\subsection{The case of spinor Grassmannians (type $D_n$)}

In this case $Y$ parametrizes one of the two families of $n$-dimensional isotropic
subspaces of $\C^{2n}$ equipped with a non-degenerate quadratic form. It is the homogeneous space
$Y\,=\,{\rm SO}(2n)/{\rm U}(n)$.
By \cite{snow2}, the non-vanishing of $H^q(Y,\,\Omega^p_Y(l))$ amounts
to the existence of an $l$-admissible $D_n$-sequence
of weight $p$ and cohomological degree $q$ in the following sense:
\begin{definition}
Fix $n,l\in\mathbb{N}$ with $l>0$.
A $n$-uple of integers $(x_i)_{0 \leq i \leq n-1}$ will be called an
\emph{$l$-admissible $D_n$-sequence} if
\begin{itemize}
\item $|x_i| = i$ for all $0\leq i\leq n-1$,
\item $x_i + x_j \not = l$ for all $i < j$.
\end{itemize}
Its weight is defined to be $$p\,:=\,\sum_{x_i>0} x_i$$ and its cohomological degree
$q \,:=\, \# \left \{ (i,j)\,\mid\, i < j \mbox{ and } x_i + x_j > l \right \}$.
\end{definition}

\begin{remark}
\label{rema:1-Dn}
Observe that the only $1$-admissible $D_n$-sequence is the sequence $(0,-1,\ldots,-n)$ with $p=q=0$.
In fact, the $1$-admissibility condition leads to the implication $(x_v>0 \Longrightarrow x_{v-1}>0)$, and
thus for other sequences to $x_1=1$ so that $x_0+x_1=1$.
\end{remark}

\noindent
We continue to use Notation \ref{nota:x+} except that now $Q\,=\,
\{(i,j)\ \mid\ i<j \mbox{ and }x_i+x_j>l \}$.
Since $Y$ is not a projective space or a quadric, we have $n \geq 5$.
The first part of theorem~\ref{theo:general_inequalities} amounts to the following proposition in this case:
\begin{proposition}[First part of Theorem~\ref{theo:general_inequalities} for spinor Grassmannians]
\label{prop:Dn-1}
Let $(x_i)$ be an $l$-admissible $D_n$-sequence of weight $p$ and cohomological degree $q$, with $n \geq 5$. Then
$$ l+q \geq \frac{4p}{n} \ , $$
with equality occurring if and only if $x_i=i$ and $l=2(n-1)$.
\end{proposition}

\begin{proof}
First of all, if $x_{n-1} = -(n-1)$, let $x'$ be the sequence of length $n-1$ with $x_i' = x_i$ for $i \leq n-2$.
Then $x'$ is evidently $l$-admissible. It has weight $p$ and cohomological degree $q$.
By induction on $n$, we get that $l+q \,\geq\, \frac{4p}{n-1}\,>\, \frac{4p}{n}$. Thus, in the rest of the proof, we
assume that $x_{n-1} = n-1$.

Let us first assume that $l>2(n-1)$. In this case, we have $q=0$.
Since in any case $p \leq \frac{n(n-1)}{2}$, we get that $\frac{4p}{n} \leq 2(n-1) < l+q$.

Let us now assume that $n \leq l \leq 2(n-1)$. Then, we denote by $u$ the unique
integer that satisfies the following condition
$$ \# \{ i \,\mid\, x_i>0, l-n+1 \leq i \leq n-1 \} \,=\, u+1\, .$$ Since the sequence $(x_i)$ is $l$-admissible,
if for $l-n+1 \leq i \leq n-1$ we have $x_i>0$, then $l-n+1 \leq l-i \leq n-1$ and $x_{l-i}<0$.
This implies that $$n-1-(l-n+1) \geq 2u\, ,$$ that is, $l \leq 2(n-u-1)$. Since $n \leq l \leq 2n-2u-2$, we have $n \geq 2u+2$.
The sum of positive $x_i$'s with $l-n+1 \leq i \leq n-1$ can be at most
$(u+1)(n-1) - \frac{u(u+1)}{2}$. Therefore, we have
$$4p \,\leq\, 4(u+1)(n-1)-2u(u+1) + 2 (l-n+1)(l-n)\, .$$ On the other hand
$q \geq u$, so introducing
$$
\Delta \,:= \,(u+l)n - 2(l-n+1)(l-n) - 4(u+1)(n-1) + 2u(u+1)\, ,
$$
the proposition amounts to the positivity of $\Delta$ whenever $n \leq l \leq 2(n-u-1)$.

After fixing $u$ and $n$, the above defined $\Delta$ is a concave function on $l$, 
so we only need to consider the values of $\Delta$ when
$l=n$ and when $l=2(n-u-1)$. When $l=n$, we get that $$\Delta \,=\,
n^2 - (4+3u)n + 2u^2+6u+4\, .$$ Fixing $u$, the two roots of
this polynomial are $n=u+2$ and $n=2u+2$. Since we know that $n \geq 2u+2$, we have $\Delta \geq 0$ for
$l=n\leq 2(n-u-1)$. For $l=2(n-1-u)$, we have $$\Delta\,=\, 3un - 6u(u+1)$$ Since once
again $n \geq 2u+2$, we get that $\Delta \geq 0$, and hence the inequality in the
proposition follows for any $l$ such that $n \leq l \leq 2(n-u-1)$.

Moreover, we show that the equality $l+q \,= \,\frac{4p}{n}$
can only occur if $x_i\,=\,i$ and $l\,=\,2(n-1)$.
Indeed, let us assume that $\Delta\,=\,0$. By the concavity argument, we have either $l=n$ or $l=2(n-u-1)$. If $l=n$,
we also get by the above argument that $n=2u+2$. Since the inequality
$$4p \,\leq\, 4(u+1)(n-1)-2u(u+1)$$ becomes an equality, we conclude that $x$ is
of the form $(-0,-1,\cdots,-u,u+1,u+2,\cdots,2u+1)$. This implies that $q\,=\,
\frac{u(u+1)}{2}$, and since $q=u$, we have $u=1$
and $n=4$, and the last equality contradicts the hypothesis of the proposition.
If $l=2(n-u-1)$, since $\Delta = 3un - 6u(u+1) = 0$, we have
$n=2(u+1)$ or $u=0$. The case of $n=2(u+1)$, $n=l$, was already dealt with earlier.
Thus we have $u=0$ and $l=2(n-1)$. The equality
$$4p = 4(u+1)(n-1)-2u(u+1) + 2 (l-n+1)(l-n)$$ amounts to $p = \frac{(n-1)n}{2}$, so that
$x_i = i$ for all $i$, and we are in the case of the proposition.

Let us now assume that $l<n$. We consider the sequence $(x'_i)$ with $x'_i = x_i$
for $i<n-1$ and $x'_{n-1} = -(n-1)$. We observe
that $(x'_i)$ is $l$-admissible with weight $p'=p-(n-1)$ and cohomological degree 
$q'$ satisfying $q' \leq q - (n-l)$. In fact, $x_i + x_{n-1} > l$ for $i<n-1-l$, and 
$x_{n-1-l}=n-1-l$ by $l$-admissibility, so that $x_{n-1-l} + x_{n-1} > l$.

By our very first argument, we have $\frac{4p'}{n-1} < l+q'$, so that
$\frac{4p}{n} < l+q'+4$. Therefore, if $q'\leq q-4$, then we are done. This is indeed the case 
if $n-l \geq 4$. Thus, we assume that $q' \geq q-3$, and
so $n \leq l+3$. We now consider these cases.

If $n=l+3$, we have $x_{n-1} = l+2$, and so $x_2=2$. Since $q' \geq q-3$, we get
that $x_i<0$ for $3 \leq i \leq n-2$. Thus we have $p \leq l+5$. The inequality in the
proposition is implied by the inequality
$l+3\,>\, \frac{4(l+5)}{l+3}$, which in turn is true for $l\,\geq\, 3$. Observe that the value
$l\,=\,2$ is excluded because we would
then have $x_4\,=\,4$. Therefore, either $x_2+x_0=2$ (if $x_2=2$) or $x_2+x_4=2$ (if $x_2=-2$).

If $n=l+2$, then there is at most one integer $i$ such that $2 \leq i \leq n-2$ and 
$x_i>0$. By admissibility, $x_1=1$,
and therefore $x_{l-1}=-l+1$. Moreover, we have $x_l=-l$. This implies that $p \leq 
2l$. The inequality of the proposition is implied by
the inequality $l+2 > \frac{8l}{l+2}$, which in turn is true for $l \geq 3$.

The value $n=l+1$ would contradict $l$-admissibility, since we would then have $x_l=l$.
\end{proof}

We now prove the second part:
\begin{proposition}[Second part of Theorem~\ref{theo:general_inequalities} for spinor
Grassmannians]
\label{prop:Dn-2}
Let $(x_i)$ be an $l$-admissible $D_n$-sequence of weight $p$ and cohomological degree $q$, with $n \geq 5$
and $q>0$. Then
$$ l+q \leq p \ . $$
Moreover, the equality $p=q+l$ holds if and only if
there are exactly two indices $i,j$ such that $x_i>0,x_j>0$ and they satisfy the
condition that either
$x_i+x_j=l+1$ (in this case $q=1$) or
$x$ is equal to $(0,1,-2,-3,\ldots,-l,l+1,-(l+2),-(l+3),\ldots,-n)$ (then $q=2$).
\end{proposition}

\begin{proof}
Let $j$ be the smallest integer such that there exists $i<j$ with $x_i+x_j>l$.
Observe that $x_j>0$, and by $D_n$-admissibility, $x_0+x_j\neq l$ so that $j \neq l$.
We first deal with the case $j>l$.
The argument in this case is similar to the case of type $C_n$:
$$
\begin{array}{rcl}
  q=\#\, Q & = & \#\, Q(v=j,j-l\leq u<j) + \#\, Q(v=j,u<j-l) + \#\, Q(v>j) \\
  & \leq & \sum_{j-l\leq u <j} x_u^+ + (j-l) + \sum_{v>j} x_v^+ \\
  & \leq & \sum_{u<j} x_u^+ + x_j^+ - l + \sum_{v>j} x_v^+ = p-l\ .
\end{array}
$$
If under the assumption $j>l$ the equality $p=q+l$ holds, then
by the second inequality we have $x_u\leq 0$ for $u<j-l$,
and by the first inequality we have $x_u^+\leq 1$ for $j-l\leq u<j$.
If $x_{j-l}=-(j-l)$, then $x_{j-l}+x_j=l$, contradicting $l$-admissibility.
Therefore, $x_{j-l}=j-l\leq 1$, so $j-l=1$ and $x_1=1$. When $v>j$,
the first inequality leads to the implication $(x_v>0 \Longrightarrow \forall~ u<v,
x_u+x_v>l)$. Assuming the existence of a $v>j$ such that $x_v>0$, we get
that $x_{v-1}>0$ and $x_{v-2}>0$ because $l>1$
(see Remark \ref{rema:1-Dn}).
By descending induction, this would lead to $x_{j-1}>0$.
Then $j-1=1$, hence $j=2$, $l=1$, which is a contradiction. 
Thus, if $v>j$, then $x_v<0$. Hence $x=(0,1,-2,-3,\cdots,-l,l+1,-(l+2),-(l+3),\ldots,-n)$
(with $p=l+2$ and $q=2$).

Let us now assume that $j<l$, and let $i$ be the largest integer such that $i<j$ and
$x_i+x_j>l$. Note that $x_i>0$ and $x_j>0$, and by maximality of $i$
we have $x_k<0$ for $i<k<j$.
We have
$$
\begin{array}{rcl}
  q=\#\, Q & = & 1 + \#\, Q(v=j,u<i) + \#\, Q(v>j) \\
  & \leq & 1 + \sum_{u<i} x_u^+ + \sum_{v>j} x_v^+ \\
  & \leq & (x_i+x_j-l) + \sum_{u<i} x_u^+ + \sum_{v>j} x_v^+ = p-l \ .
\end{array}
$$
The inequality $l+q\leq p$ is proved. 
In the case of $q=p-l$, for $1\leq u<i$ we have $x_u^+\leq 1$
by the first inequality. By the second inequality we have
$x_i+x_j=l+1$. Using the descending induction argument, if there is a 
$v>j$ such that $x_v>0$, then $x_l>0$, and $x_0+x_l=l$ contradicting the admissibility.
The only positive entries are among $x_1,x_i,x_j$. In this case, since $i+j=l+1$,
we have $Q=\{(i,j)\}$, so $q=1$, $p=l+1$ and $x_1<0$.
\end{proof}

\subsection{The exceptional cases (type $E_6$ or $E_7$)}

Now $Y$ is homogeneous under a group of type $E_6$ (case $EIII$) or $E_7$ (case $EVII$).
In the first case, we have $\dim(Y)=16$ and $c_1(Y)=12$. In the second case, we have
$\dim(Y)=27$ and $c_1(Y)=18$. These values are well-known to the specialists; several arguments
for the computation of $c_1$ can be found at the end of Section 2.1 in \cite{cmp}.

{}From Tables 4.4 and 4.5 in \cite{snow2} we conclude that the inequalities we are looking
for hold:
\begin{proposition}[Theorem~\ref{theo:general_inequalities} for the exceptional cases]
\label{prop:E}
Let $Y$ be a Hermitian symmetric space of type $EIII$ or $EVII$.
Let $l,p,q$ be integers with $l>0,p>0$, and such that $$H^q(Y,\Omega_Y^p(l)) \neq 0\, .$$
Then, $p\, \frac{c_1(Y)}{\dim(Y)} \leq l+q$. Equality implies that $p=\dim(Y),l=c_1(Y)$
and $q=0$.

Assume moreover that $q>0$. Then $l+q \leq p$. 
\end{proposition}

\subsection{A cohomological property}

We will need the following corollary of Theorem~\ref{theo:general_inequalities} to prove our stability results.
\begin{proposition}
\label{prop:general-ineq}
Let $Y$ be a compact irreducible Hermitian symmetric space, but not a projective space.
Let $l,p,q$ be integers, with $q<\dim Y$, such that
$H^q(Y,\Omega^p_Y(l))\neq 0$.
Then,
$$ l+q \,\geq\, p \frac{c_1(Y)}{\dim(Y)}\, ,$$ with equality
holding if and only if 
\begin{itemize}
\item $p=\dim(Y),q=0$ and $l=c_1(Y)$, 
\item or $p=q=l=0$,
\item or $Y$ is a quadric and $l=0$,
\item or $Y \simeq \Q^4,l=2,p=3,q=1$.
\end{itemize}
\end{proposition}

\begin{proof}
Assume that $H^q(Y,\Omega_Y^p(l))\not=0$.
If $l \geq 0$, then we are done by Theorem \ref{theo:general_inequalities}(1).

If $p\,=\,0$, then $q\,=\,0$ if
$l\,>\,0$, while $q\,=\,\dim(Y)$ if $l\,<\,0$. Thus this case is also settled.

Assume that $q>0$ and $l >0$. Let us prove that actually
\begin{equation}
\label{equa:inegalite_poincare}
l+q \,<\, p \frac{c_1(Y)}{\dim(Y)} + \dim(Y) - c_1(Y)\, .
\end{equation}
Firstly, the theorem states that $l+q \,\leq \,p$, and we have $p\leq\dim(Y)$. Since
$$(1-\frac{c_1(Y)}{\dim(Y)})(l+q) \leq \dim(Y) - c_1(Y)\, ,$$ we find
$$
l+q \leq \frac{c_1(Y)}{\dim(Y)}(l+q) + \dim(Y) - c_1(Y) \leq p \frac{c_1(Y)}{\dim(Y)} +\dim(Y) - c_1(Y)\ .
$$
Note that the equality here would imply that $p=\dim(Y)$, in which case, by Kodaira vanishing, we cannot have $q>0$. 
Thus, the inequality (\ref{equa:inegalite_poincare}) is proved.
Now, coming back to the proof of Proposition~\ref{prop:general-ineq}, if $l\,<\,0$
then Serre duality leads to
$$H^{\dim(Y)-q}(Y,\Omega_Y^{\dim(Y)-p}(-l))\,\not=\,0\, .$$ 
The relation (\ref{equa:inegalite_poincare}) gives that
$(-l)+(\dim(Y)-q) < (\dim(Y)-p)\frac{c_1(Y)}{\dim(Y)} + \dim(Y) - c_1(Y)$, or in other
words,
\begin{eqnarray*}l+q \,>\, p\frac{c_1(Y)}{\dim(Y)}\, .\qedhere\end{eqnarray*}
\end{proof}

\section{Restrictions with small Picard group}
\label{section:hermitian}

\subsection{Short split resolution}

We will prove stability of the restriction of $\Omega_Y$ to subschemes whose structure
sheaf has a short split resolution in the following sense:

\begin{definition}
\label{defi:short-resolution}
A subscheme $X \subset Y$ is said to have a \emph{short split resolution} if there is a resolution
$$0\to\F_k\to\F_{k-1}\to\cdots\to\F_1\to\F_0\to\Oc_X\to 0\, ,$$
where $\F_0 = \Oc_Y$, $\F_i:=\oplus_j\Oc_Y(-d_{ij})$, and
the length $k$ of the resolution satisfies $k < \dim(Y)$.
\end{definition}

\begin{example}
\label{exam:complete-intersection}
The Koszul resolution of complete intersections is a short split resolution for a positive-dimensional
complete intersection in $Y$. If, moreover, none of the equations are linear, 
then the integers $d_{ij}$ in Definition~\ref{defi:short-resolution} satisfy $d_{ij} \geq i+1$.
\end{example}

Our second class of examples are some arithmetically Cohen-Macaulay subschemes. 
Let $X \subset \P^N$ be a subscheme defined by a homogeneous ideal $J$ in
the homogeneous coordinate ring $A:=\C[X_0,\ldots,X_N]$. Recall that $X$ is called \emph{\acm}
if the depth of $A/J$ is equal to the dimension of $A/J$, namely $\dim X + 1$. 
Let, moreover, $I \subset A$ denote the homogeneous ideal of $Y$. 
The reason why we will consider \acm subschemes is the following:
\begin{lemma}\label{fact:acm}
Let $X$ be an \acm subscheme of $Y$ defined in $\P^N$ by an ideal $J$.
Assume that $A/J$ has finite projective dimension over $A/I$.
Then, the structure sheaf $\Oc_X$ has a resolution by split vector bundles over $Y$ of length $k=\dim Y - \dim X$
$$0\to\F_k\to\F_{k-1}\to\cdots\to\F_1\to\F_0\to\Oc_X\to 0 ,$$
where $\F_0 = \Oc_Y$ and $\F_i:=\oplus_j\Oc_Y(-d_{ij})$ with $d_{ij}\geq i$ for $i>0$.
In particular, if $\dim(X)>0$, then $X$ has a short split resolution.
\end{lemma}

\begin{proof}
We have $I \subset J \subset A$. By Auslander-Buchsbaum formula~\cite[Theorem 19.1]{matsumura},
we have the equality
$$\pd_{A/I}(A/J) = \depth(A/I) - \depth(A/J).$$
Moreover, any homogeneous space embedded by a homogeneous ample line bundle is \acm 
(see for example \cite[Corollary 3.4.4]{bk}). Thus,
$$\depth(A/I) = \dim(A/I) = \dim Y + 1.$$ Therefore,
$\pd_{A/I}(A/J) = \dim Y - \dim X$. Hence a minimal free resolution of $A/J$ over $A/I$ has length $k=\dim Y - \dim X$.
\end{proof}

\begin{remark}
 Concretely, our assumption that $A/J$ has finite projective dimension over $A/I$ might be difficult to check.
 For example, Theorem \ref{theo:restriction_hermitian} becomes obviously wrong if we take $X \simeq \P^1$
 (since the restriction of $\Omega_Y$ to $X$ will be split in this case).
 All our arguments hold for $X \simeq \P^1$, except that in this case the corresponding projective dimension is infinite.
\end{remark}

\begin{lemma}
 \label{lemm:degree-resolution}
 Let $X \subset Y$ admitting a short split resolution. Then, there is a short split resolution such that
 the integers $d_{ij}$ in Definition \ref{defi:short-resolution} satisfy $\forall i,j, \ d_{ij} \geq i$.
 If, moreover, $X$ is not linearly degenerate in $\P^N$, then we may assume that
 the inequalities $d_{ij}\geq i+1$ hold for all $(i,j)$.
\end{lemma}
\begin{proof}
 We consider as in the proof of Fact \ref{fact:acm} a minimal free resolution of $A/J$ over $A/I$.
 The hypothesis implies that such a resolution will have length strictly less than $\dim(Y)$.
 Moreover, since for such a resolution the differentials
 have positive degree, we have $d_{ij}\geq i$ for all $i,j$.
 If $X$ is not included in any hyperplane, it has no equation of degree $1$, 
 so $d_{1j} \geq 2$ for all $j$, and we deduce that $d_{ij} \geq i+1$. 
\end{proof}

\subsection{General argument}
Recall that the \emph{slope} of a coherent torsion-free sheaf $F$ of positive rank on a polarized manifold $(X,L)$ is 
$$\mu(F):=\frac{c_1(F)\cdot L^{\dim X-1}}{\rank F}.$$
Recall that a coherent torsion-free sheaf $E$ on a polarized manifold $(X,L)$ is said to be \emph{(slope-)stable}
if the slope of all its subsheaves of positive smaller rank is less than its slope.
Note that it is enough to check the inequalities for saturated subsheaves.

\label{subsec:general}
We can now prove the
\begin{theorem}
\label{theo:restriction_hermitian}
Let $Y$ be any compact irreducible Hermitian symmetric space
excluding a projective space and a quadric. 
Let $X$ be a locally factorial positive dimensional subvariety of $Y$ 
having a short split resolution and such that $\Pic(X)\,=\,\Z\cdot \Oc_Y(1)_{| X}$.
Then the restriction of $\Omega_Y$ to $X$ is stable.
\end{theorem}

\noindent
Note that if $X$ is a complete intersection of $\dim X\,\geq\, 3$, the constraint on the Picard group of the complete
intersection is ensured by Lefschetz theorem~\cite[Example 3.1.25]{lazarsfeld}.

\begin{proof}
We will later prove a slightly weaker result for quadrics and projective spaces
(Theorem \ref{theo:restriction_exceptions}), thus, for the moment $Y$ is any
compact irreducible Hermitian symmetric space.

We first explain how, building on a classical argument, the assumption on the existence of a small split resolution
reduces the check of the stability inequalities to vanishing theorems.
Let $\F$ be a coherent subsheaf of $\Omega_{Y|X}$ of rank $0\,<\,p\,<\,\dim Y$. 
Since $X$ is assumed to be locally factorial, 
the rank one reflexive subsheaf $\det\F\,:=\, (\bigwedge\nolimits^p \F)^{**}$
of $\bigwedge\nolimits^p\Omega_{Y|X}$ is invertible~\cite[Proposition 1.9]{hartshorne}
and hence isomorphic to $\Oc_Y(-d)_{\mid X}=:\Oc_X(-d)$ for some integer $d$. We have
\begin{eqnarray*}
\mu(\OY_{\mid X})&=&\frac{\Oc_X(1)^{\dim(X)-1}\cdot K_Y}{\rank(\Omega_{Y|X})}=-\frac{c_1(Y)}{\dim(Y)}\cdot \deg_Y X\deg Y\\
\mu(\F)&=&\frac{\Oc_X(1)^{\dim(X)-1}\cdot \det\F}{\rank(\F)}=-\frac{d}{p}\cdot \deg_Y X\deg Y.
\end{eqnarray*}
The inclusion $\F\,\subset\,\Omega_{Y|X}$ yields the non-vanishing of
$$H^0(X,Hom(\det\F,\Omega^p_{Y|X}))\,=\,H^0(X,\Omega^p_{Y}(d)_{\mid X}),$$ 
from which we have to deduce the stability inequality
$$\mu(\F)\,<\,\mu(\Omega_{Y|X})\, , \textrm{ equivalently, } d\,>\,p \, \frac{c_1(Y)}{\dim(Y)}.$$

Consider a resolution of $\Oc_X$ as in Definition~\ref{defi:short-resolution}. The resolution
$$0\to\F_k\otimes\Omega^p_{Y}(d)
\to\cdots \to\F_1\otimes\Omega^p_{Y}(d)\to\F_0\otimes\Omega^p_{Y}(d)\to\Omega^p_{Y}(d)_{\mid X}\to 0$$
translates the non-vanishing of $H^0(X,\Omega^p_{Y}(d)_{\mid X})$ 
into the non vanishing of one of the cohomology groups in the decomposition
$$H^i(Y,\F_i\otimes\Omega^p_{Y}(d))=\oplus_jH^i(Y,\Omega^p_{Y}(d-d_{ij})),$$
say of $H^i(Y,\Omega^p_{Y}(d-d_{ij}))$.

We now assume that $Y$ is not a projective space.
In our setting Proposition~\ref{prop:general-ineq} reads
\begin{equation}
\label{equa:dJ}
( d-d_{ij} ) +i \geq p \frac{c_1(Y)}{\dim(Y)}.
\end{equation}
It follows that $d\geq p \frac{c_1(Y)}{\dim(Y)}$.

If the equality $d= p \frac{c_1(Y)}{\dim(Y)}$ holds, we get that $d_{ij}=i$,
and the equality in (\ref{equa:dJ}) holds. Now assume that $Y$ is not a quadric.
Therefore, Proposition~\ref{prop:general-ineq} gives that $p=0$ or $p=\dim(Y)$,
equivalently, as we may assume $\F$ saturated (i.e. with torsion free quotient)
either $\F=\{0\}$ or $\F=\Omega_Y$. Thus, $\Omega_Y$ is stable.
\end{proof}

Some remarks regarding the two excluded cases in Theorem~\ref{theo:restriction_hermitian}.
\begin{remark}
If $X$ is contained in a linear subspace $H$ in a projective space $Y$, then for the exact
sequence, 
$$ 0 \to {N^*_{H|Y}}_{\mid X} \to \Omega_{Y|X} \to {\Omega_H}_{\mid X} \to 0\, ,$$
the slope $\frac{-\deg X}{\codim_Y H}$ of ${N^*_{H|Y}}_{\mid X}$ 
is strictly bigger than the slope $\frac{-(\dim Y+1)\deg X}{\dim Y}$ of $\Omega_{Y|X}$. Thus,
$\Omega_{Y|X}$ is not even semi-stable.
\end{remark}

\begin{remark}\label{rema:linear-unstable}
If $X \subset Y$ is a linear section of the quadric $Y$ (i.e. $\deg_YX=1$), then, as the above proof shows,
the restriction of $\Omega_Y$ to $X$ is semi-stable.
Furthermore, all the vector bundles in the following exact sequence
$$ 0 \to N^*_{X|Y} \to \Omega_{Y|X} \to \Omega_X \to 0$$
have equal slope $-1$. So $\Omega_{Y|X}$ is not stable. 
In fact, let us continue the end of the argument in the proof of Theorem~\ref{theo:restriction_hermitian}:
we have in this case $i=1$ and $d_{11}=1$ (since a resolution of $X$
contains only one term of degree $1$, $X$ being a hyperplane section). We get
$H^1(Y,\Omega^p_{Y}(d-1)) \neq 0$, which implies that $p=1$ and $d-1=0$ by Proposition~\ref{prop:general-ineq}.
Thus the only destabilizing subsheaf is $\Oc_X(-1)$.
\end{remark}

Thus, to get a result similar to Theorem \ref{theo:restriction_hermitian} in these
two cases, we exclude the case where $X$ has a linear equation:
\begin{theorem}
\label{theo:restriction_exceptions}
Let $Y$ be a smooth quadric of dimension at least $3$ or a projective space. Then $\Omega_Y$ is stable.
Let $X$ be a locally factorial subvariety in $Y$ having a short split resolution and
such that $\Pic(X)\,=\,\Z\cdot\Oc_Y(1)_{| X}$. Assume that $X$ is contained in no
hyperplane section of $Y$.
Then the restriction of $\Omega_Y$ to $X$ is stable.
\end{theorem}

\noindent 
Note that by Lefschetz theorem, this theorem applies to very general cubic hypersurfaces of~$\Q^3$.

\begin{proof}
We continue with the notation of Theorem \ref{theo:restriction_hermitian}.
If $Y$ is a quadric, by the above proof of Theorem \ref{theo:restriction_hermitian}, we have
the non-vanishing of some $H^i(Y,\F_i\otimes\Omega^p_{Y}(d))$. If $i>0$, by (\ref{equa:dJ})
we have for some $j$ the inequality
$$d - d_{ij} + i \geq p\frac{c_1(Y)}{\dim(Y)}~~(=p)\, .$$ Since, by Lemma \ref{lemm:degree-resolution},
we get $d_{ij} > i$, we conclude that $d>p$, as wanted.
If $i=0$ and we assume $d\leq p$ by contradiction, then $H^0(Y,\Omega^p_Y(d)) \neq 0$, and by a result due to
Snow (see Section~\ref{sec:quadric}), we get that either $p=0$ or $p=\dim(Y)$. 
This implies stability as in the proof of Theorem~\ref{theo:restriction_hermitian}.

Assume now that $Y$ is the projective space $\P^n$.
We may assume that $0<p<n$. We wish to prove that $$\frac dp \,>\, \frac{n+1}{n}\, .$$
Since $\frac{p+1}{p} \,>\, \frac{n+1}{n}$, it is enough to prove that $d \,\geq\, p+1$.
For integers $p,q,l$, we have
$H^{q}(\P^n,\Omega^p_{\P^n}(l)) \neq 0$ if and only if one of the following hold:
\begin{enumerate}
\item $l>0,p<l$ and $q=0$,
\item $l=0$ and $p=q$,
\item $l<0,n-p<-l$ and $q=n$.
\end{enumerate}
Once again, we get that $H^i(Y,\F_i \otimes \Omega^p_{Y}(d)) \neq 0$ for some $i$.
If $i=0$, since $\F_0=\Oc_Y$, this implies $d>p$ or $p=d=0$.
If $i>0$, since $i<n$, this implies that $i=p$ and $d=d_{ij}$ for some $j$. Thus
we have $d \geq i+1 = p+1$, as wanted.
\end{proof}

\section{Restriction to a hypersurface with an increase of the Picard group}
\label{section:small dimensions}


\subsection{Another argument for general complete intersection}
In this section, we want to get rid of the assumption on the Picard group.
This can be done at the cost of considering only general complete intersections.
We get the following adaptation of theorems \ref{theo:restriction_hermitian} and \ref{theo:restriction_exceptions}.
\begin{theorem}\label{theo:general-restriction_divisor}
Let $Y$ be a compact irreducible Hermitian symmetric space. Let $X$ be a general positive-dimensional complete intersection in $Y$.
If $Y$ is neither a projective space nor a quadric, then the restriction of
$\Omega_Y$ to $X$ is semi-stable.

If $Y$ is a smooth quadric or a projective space,
assume that none of the hypersurfaces $H_i$ is linear.
Then the restriction of $\Omega_Y$ to $X$ is semi-stable.
\end{theorem}

\begin{proof}
Let $V = \Gamma(Y,\Oc_Y(1))^*$ be the minimal homogeneous embedding of $Y$, so that $Y \subset \P V$.
Let $$S = \P S^{h_1}V^* \times \ldots \times \P S^{h_c}V^*.$$ Let
$Z \subset Y \times S$ be the universal
family of complete intersections defined by
$$(x,([H_1],\ldots,[H_c])) \in Z \iff \forall i, H_i(x)=0\, .$$ Thus we have morphisms $p:Z \to Y$ and $q:Z \to S$ such
that for general $s=(H_i) \in S$ (i.e. for $s$ in a non empty Zariski open set), 
the inverse image $X_s:=q^{-1}(s)=\cap_{i=1}^c (H_i=0)$ is a complete intersection in $Y$ of multi-degree $(h_i)$.

To proceed by contradiction, assume that semi-stability of the restriction of $\Omega_Y$ to $X$
does not hold.
We will use the relative Harder-Narasimhan filtration relative to $q:Z \to S$ \cite[Theorem 2.3.2]{hl}
(the idea of using this relative version appears e.g. in the proof of \cite[Theorem 7.1.1]{hl}). 
After the choice of a suitable birational projective morphism $f:T \to S$, we can build the following commutative diagram
$$
\xymatrix{
g^*p^* \Omega_Y \ar[d] & p^* \Omega_Y \ar[d] \\ 
Z_T \ar[d] \ar[r]^g & Z \ar[d]_q \ar[r]^p & Y \\
T \ar[r]_f & S.
}
$$
Here, $Z_T$ is the fibered product of $Z$ and $T$.
The pulled-back sheaf $g^*p^* \Omega_Y$ has a filtration
which induces for a general point $s \in S$ the Harder-Narasimhan filtration of $\Omega_{Y|X_s}$. 
We denote by $\F$ the first term of this filtration and by $k$ its rank. 
The assumption that semi-stability fails amounts to saying that $0<k<\dim Y$.
The rank one reflexive
subsheaf $\det\F\,:=\, (\bigwedge\nolimits^k \F)^{**}$
of $p^* \bigwedge\nolimits^k\Omega_{Y}$ is invertible. Since $S$ is smooth, $f$ is an isomorphism in codimension 1;
so the same holds for $g$, and since $Z$ is also smooth, the line bundle $\det\F$ on $Z_T$ defines a line bundle on $Z$ denoted by $\cL$.
It is a subsheaf of $\bigwedge^k p^* \Omega_Y$.
Now, $p$ is a locally trivial morphism with fibers isomorphic to
products of projective spaces, so $\Pic(Z) \simeq \Pic(Y) \times \Pic(S)$, and $\cL$ can be expressed as
$p^* \cL_Y \otimes q^* \cL_S$, for some line bundles $\cL_Y \in \Pic(Y)$ and $\cL_S \in \Pic(S)$.
This is the main feature of considering families: whereas the determinant of a single destabilizing subsheaf
may fail to lie on $\Pic (Y)_{\mid X}$, the determinant of the first term of the relative Harder-Narasimhan filtration does.

Let $d$ be the integer such that $\cL_Y \simeq \Oc_Y(-d)$, and let $X=X_s$.
Given $s \in S$, we have $\cL_{|X \times \{s\}} \simeq \cL_{Y|X}$, and hence
for general $s \in S$, this yields an injection of sheaves $\Oc_Y(-d)_{|X} \subset p^* \Omega_Y$.

Let $h = h_1 \ldots h_c$, we have:
\begin{eqnarray*}
\mu(\OY_{|X})&=&\frac{\Oc_X(1)^{\dim(X)-1}\cdot K_Y}{\rank(\Omega_{Y|X})}=-\frac{c_1(Y)}{\dim(Y)}\cdot h \cdot \deg Y\\
\mu(\F_{|X})&=&\frac{\Oc_X(1)^{\dim(X)-1}\cdot \det\F_{|X}}{\rank(\F)}=-\frac{d}{p}\cdot h \cdot \deg Y\, .
\end{eqnarray*}
Since $\Oc_Y(-d)_{|X}\,\subset\,\Omega_{Y|X}$, it follows that
$$H^0(X,Hom(\Oc_Y(-d)_{|X},\Omega^k_{Y|X}))\,=\,H^0(X,\Omega^k_{Y|X}(d))$$ does not vanish.
Using this and previous results, valid on $\Pic (Y)_{\mid X}$, we deduce the inequality
$$\mu(\F_{|X})\,<\,\mu(\Omega_{Y|X})\, , \ \ i.e.,\ \ d\,>\,k \, \frac{c_1(Y)}{\dim(Y)}\ .$$
This contradicts the construction of $\cL$ as the determinant of the first term of the relative
Harder-Narasimhan filtration.
\end{proof}

\subsection{Variation on a result due to Langer}

To state a general theorem due to Langer, we only assume in this subsection that $Y$ is a smooth projective variety.
We denote by $\Oc_Y(1)$ a polarization of $Y$. We set $d:=c_1(\Oc_Y(1))^{\dim Y}$

\begin{definition}
 \label{defi:Delta}
The discriminant of a rank $r$ vector bundle $E$ on $Y$ is $$\Delta(E):=\left[2rc_2(E)-(r-1)c_1^2(E)\right]\cdot c_1(\Oc_Y(1))^{\dim Y-2}.$$
\end{definition}

\begin{theorem*}[{\cite[Theorem 5.2]{la}}]
Consider a smooth projective variety $Y$.
Consider a $\Oc_Y(1)$-stable vector bundle $E$ of rank $r$ on $Y$.
Let $X$ be a smooth divisor in the complete linear system $\vert\Oc_Y(h)\vert$. 
If
\begin{equation}
 \label{inequality_langer}
 h>\frac{r-1}{r}\Delta(E)+\frac{1}{dr(r-1)}\, ,
\end{equation}
then $E_{\mid X}$ is ${\Oc_Y(1)}_{\mid X}$-stable.
\end{theorem*}

\begin{remark}\label{Langer-bounds}
As noticed in \cite[Remark 5.3.2]{la}, when $r>2$, the inequality (\ref{inequality_langer}) is equivalent
to the inequality $h>\frac{r-1}{r}\Delta(E)$, since $h$ is an integer. Langer also points out that
his Theorem ``can be further improved at the cost of simplicity''. This is what is done in
Proposition \ref{prop:liu}.

The proof of the following result was communicated to us by Jie Liu:

\begin{proposition}
 \label{prop:liu}
 Consider a smooth projective variety $Y$.
 Consider a $\Oc_Y(1)$-stable vector bundle $E$ of rank $r$ on $Y$.
 Let $X$ be a smooth divisor in the complete linear system $\vert\Oc_Y(h)\vert$. 
 Let $m$ be a common divisor of $r$ and the degree of $E$. If
 \begin{equation}
 \label{inequality_liu}
 h>\frac{r-1}{rm}\Delta(E)+\frac{m}{dr(r-1)}\, ,
 \end{equation}
 then $E_{\mid X}$ is ${\Oc_Y(1)}_{\mid X}$-stable.
\end{proposition}
\begin{proof}
In the proof of {\cite[Theorem 5.2]{la}}, the term $\frac{1}{r(r-1)}$ that bounds the difference of two slopes
(namely in Langer's notation $\mu_{\max}(G)-\mu(G)$)
may be replaced by $\frac{m}{r(r-1)}$.
The inequality 
$$0\leq d\Delta (E)-\rho(r-\rho)d^2h^2+r^2(\frac{r-\rho}{r}dh-\frac{1}{r(r-1)})(\frac{\rho}{r}dh-\frac{1}{r(r-1)})$$
where $\rho$ is the rank of a putative maximally destabilizing subsheaf of the restriction $E_{\mid X}$,
may here be replaced by
$$0\leq d\Delta (E)-\rho(r-\rho)d^2h^2+r^2(\frac{r-\rho}{r}dh-
\frac{m}{r(r-1)})(\frac{\rho}{r}dh-\frac{m}{r(r-1)})$$
that is
$$0\leq d\Delta(E)-\frac{rmdh}{r-1}+\frac{m^2}{(r-1)^2}.\qedhere$$
\end{proof}

In particular, if $Y=\Q^3$ and $E=\Omega_Y$, then
$d=c_1(\Oc_{\Q^3}(1))^3=2$, $r=3$, $$\Delta(\Omega_{\Q^3})
=[6\times 4 c_1(\Oc_{\Q^3}(1))^2-2 ((-3)c_1(\Oc_{\Q^3}(1)))^2]\cdot c_1(\Oc_{\Q^3}(1))=12$$ (here we
have used the Euler sequence on $\P^4$ and the normal sequence for $\Q^3$ in $\P^4$
to compute $c_2(\Omega_{\Q^3})=4 c_1(\Oc_{\Q^3}(1))^2$).
The inequality (\ref{inequality_liu}) reads $h>\frac83 + \frac14$.
The inequality $h\geq 3$ for $h=\deg_Y X$ is therefore enough to derive the stability of
the restriction ${\Omega_{\Q^3}}_{\mid X}$.

Hence, the bounds of Langer are
\begin{itemize}
\item for $\P^2$, $h>2$,
\item for $\P^3$, $h>8/3$,
\item for $\Q^3$, $h\geq 3$.
\end{itemize}
\end{remark}

\subsection{Optimal bounds}

With some obvious exceptions, we get the stability of $\Omega_{Y|X}$:
\begin{theorem}\label{theo:restriction_divisor}
Let $Y$ be a compact irreducible Hermitian symmetric space of dimension $2$ or $3$.
Let $X \subset Y$ be a smooth divisor.
\begin{itemize}
\item Take $Y=\P^2$. Then $\Omega_{Y|X}$ is semi-stable if $\deg_Y X \geq 2$.
If $\deg_Y X \geq 3$, then $\Omega_{Y|X}$ is stable.
\item If $Y=\P^3$, assume that $\deg_Y X \geq 2$. Then $\Omega_{Y|X}$ is stable.
\item If $Y=\Q^2$, then $\Omega_{Y|X}$ is semi-stable but not stable.
\item Take $Y=\Q^3$. If $\deg_Y X=1$, then $\Omega_{Y|X}$ is semi-stable with respect to the anti-canonical polarization.
If $\deg_Y X\,\geq\, 2$, then $\Omega_{Y|X}$ is stable.
\end{itemize}
\end{theorem}
These cases will be considered in the next subsections.

\subsection{The case of $\P^2$}
Recall the Euler sequence on $\P^2$
\begin{eqnarray*}
0\,\longrightarrow\,\Omega_{\P^2}(1)\,\longrightarrow\, \OP^{\oplus 3}\,\longrightarrow\,
\Oc_{{\P^2}}(1)\,\longrightarrow \,0\, .
\end{eqnarray*}
For a smooth conic $C$, since each section of $H^0(C,\Oc(1)_{\mid C})$
is a restriction of a section on $\P^2$,
the rank $2$ vector bundle $\Omega_{\P^2}(1)_{\mid C}$ of degree $-2$ has no sections.
Therefore, $\Omega_{\P^2}(1)_{\mid C}$ is isomorphic to the direct sum of two line bundles of degree $-1$ on the rational curve $C$.
Consequently, it is semi-stable and not stable.

Let $C$ be a curve of degree $d\geq 3$ in ${\P^2}$. 
By Langer theorem (see Remark \ref{Langer-bounds}) we get the stability of $\Omega_{\P^2}(1)_{\mid C}$.

\subsection{The case of $\Q^2$}

We consider a non degenerate quadric $Q$ in $\P^3$. Recall that it is isomorphic to $\P^1\times\P^1$
in such a way that $\Oc_{\P^3}(1)_{\mid Q}=\Oc_{\P^1\times\P^1}(1,1)$. Note that the tangent bundle $TQ=\Oc(2,0)\oplus\Oc(0,2)$, 
being the sum of two line bundles of the same $\Oc_{\P^3}(1)_{\mid Q}$-degree, is $\Oc_{\P^3}(1)_{\mid Q}$-semi-stable. 
Its restriction ${TQ}_{\mid X}$ to a smooth curve $X\subset Q$ in any linear system $|\Oc_{\P^3}(d)_{\mid Q}|=|\Oc(d,d)|$
is the sum of two line bundles of degree $\Oc(d,d)\cdot \Oc(2,0)=\Oc(d,d)\cdot \Oc(0,2)=2d$.
Hence $TQ_{\mid X}$ is semi-stable for any $d\geq 1$.

\subsection{The case of $\Q^3$}

We now consider a non degenerate quadric $Q$ in $\P^4$ and a smooth hypersurface $S\subset Q$
of degree $d$. By the results in Section~\ref{sec:quadric}, the vector bundle $TQ$ is stable.
The bound in Langer's theorem computed in Remark~\ref{Langer-bounds} is $3$.
Therefore, we conclude that the restriction $TQ_{\mid S}$ is stable if $d\, \geq\, 3$.

We will study the degree $1$ and $2$ cases in the rest of this subsection.

\subsubsection{Linear sections}

For $d=1$, the isomorphism ${\rm Pic}(S)=\Z^2$ is due to the product structure $S\simeq 
\P^1\times\P^1$.

\begin{proposition}\label{prop:Q3-deg1}
If $S$ is a smooth linear section of the solid quadric $Q$, then $TQ_{|S}$ is semi-stable but not stable.
\end{proposition}
\begin{proof}
Consider a putative destabilizing sheaf $\F \subset TQ_{|S}$. Assume that the rank of $\F$ is one.
Replacing $\F$ by its reflexive hull, we get an exact sequence
\begin{equation}
\label{equa:sequence}
0 \to L \to TQ_{|S} \to E
\end{equation}
with $L$ a line bundle and $E$ a rank 2 vector bundle. Moreover, we have $\deg(\F)=\deg(L)$, thus
to prove semi-stability it suffices to show that the existence of such an exact sequence 
implies that $\deg(L) \leq \mu(TQ_{|S})=2$.

Let us write $S=\P^1 \times \P^1 \stackrel{\pi_1,\pi_2}{\rightarrow} \P^1$,
and $L = \Oc(d_1,d_2) := \pi_1^* \Oc(d_1) \otimes \pi_2^* \Oc(d_2)$.
For example, for $L = \pi_1^* T\P^1$, we have an exact sequence as in (\ref{equa:sequence}),
and $L \simeq \Oc(2,0)$ so $\deg(L)=2$. In particular $TQ_{|S}$ can not be stable.
The semi-stability inequality is proved in the following Lemma \ref{lemm:subbundle-deg1}.

\begin{lemma}
\label{lemm:subbundle-deg1}
With the notation of the proof of Proposition \ref{prop:Q3-deg1}, let $\Oc(d_1,d_2)$ be a subbundle
of $TQ_{|S}$. Then, we have $d_1+d_2 \,\leq\, 2$. 
\end{lemma}

\begin{proof}
Since $L=\Oc(d_1,d_2)$ is assumed to be a subbundle of $TQ_{|S}$, there is a non 
vanishing section of $L^* \otimes TQ_{|S}$. There is an exact sequence of sections on $S$:
$$ H^0(L^* \otimes TS) \to H^0(L^*\otimes TQ_{|S}) \to H^0(L^*\otimes \Oc_S(1) )\ .$$
We have $L^* \otimes \Oc_S(1) \simeq \Oc(1-d_1,1-d_2)$ and
$L^* \otimes TS \simeq \Oc(2-d_1,-d_2) \oplus \Oc(-d_1,2-d_2)$.

Assume that $d_1>1$. Then $H^0(L^*(1)))=0$ since $\pi_{1*} L^*(1) = 0$.
Thus $$H^0(L^*\otimes TQ_{|S}) = H^0(L^* \otimes TS) = H^0(\Oc(2-d_1,-d_2))\, .$$
By the same argument, this space of sections is not equal to $\{0\}$ if and only if
$d_1 = 2$ and $d_2 \leq 0$. Moreover, there will be non vanishing sections if and only if
$d_2=0$. We get $(d_1,d_2)=(2,0)$ (so $L$ is isomorphic to $\pi_1^* T\P^1$).

Similarly, we can deal with the case $d_2>1$. In the remaining cases we indeed have
$d_1+d_2 \leq 2$ (asserted in the lemma).
\end{proof}

We now finish the proof of Proposition \ref{prop:Q3-deg1}.
Consider a rank two subsheaf $F\subset {TQ}_{\mid S}$.
Its determinant is a line subbundle of ${\wedge^2TQ}_{\mid S}$.
As $TY(-1)=\Oc(-1)^\perp/\Oc(-1)$ is self-dual, we have
$${\wedge^2TQ}_{\mid S}={\Omega_Q^2}_{\mid S}(4)={K_Q\otimes TQ}_{\mid S}(4)={TQ}_{\mid S}(1).$$ 
Hence we get an inclusion $\det F\otimes \Oc(-1)_{\mid S}\subset {TQ}_{\mid S}$.
We already checked that line subbundles do not destabilize ${TQ}_{\mid S}$. Therefore, we infer that
$2\mu(F)-2\leq \mu({TQ}_{\mid S})=2$, which implies the desired
semi-stability inequality $$\mu(F)\leq\mu({TQ}_{\mid S}).\qedhere$$
\end{proof}

\subsubsection{Quadric sections}
For $d=2$, the surface $S$, intersection of two quadrics in $\P^4$,
is a Del Pezzo surface of degree $4$ meaning $$(-K_S)\cdot (-K_S)=4$$
(see \cite[Definition 8.1.12]{dolgachev}); it is known as a Segre quartic surface.
Its Picard group ${\rm Pic(S)}$ is isomorphic to $\Z^6$ with precise generators given by the abstract description of $S$ 
as the projective plane $\P^2$ blown-up at $5$ points in general position
(see\cite[page 550]{griffiths-harris}, \cite[Proposition 8.1.25]{dolgachev}). Recall the diagram
$$\xymatrix{Bl_p(S)\ar[d]^\mu \ar@{=}[r]^\phi&\Sigma\ar@{^(-}[r]^{\iota}\ar[d]^b&\P^3\\
S\ar[r]^\pi&\P^2&}$$
where $\mu$ is the blow up of $S$ at a point $p$ on $S$ not on a line of $S$,
$\iota\circ\phi$ is given by the linear system of lines in $\P^4$ passing through $p$
(its image is a smooth cubic $\Sigma$),
$b$ is the blow up of $\P^2$ at six points, $\iota$ is given by the linear 
system of cubics in $\P^2$ passing through the blown-up six points, and
$\pi$ is gotten from $\phi$ by the universal property of blow ups.
With $E=\sum_{i=1}^5E_i$ the sum of the five exceptional lines and $\jmath :E\to S$ the natural inclusion, the main relations
among sheaves in the two descriptions of $S$ are
\begin{eqnarray*}
\Oc_{\P^4}(1)_{\mid S}=\pi^*\Oc_{\P^2}(3)\otimes\Oc_S(-E)
\end{eqnarray*}
and
\begin{eqnarray}\label{eq:cotE}
0\to\pi^*\Omega_{\P^2}\to\Omega_S\to\jmath_*\Omega_E\to 0.
\end{eqnarray} 

Before stating Proposition \ref{prop:Q3-deg2} which is the main result of this section, we start
recalling a result of Fahlaoui:

\begin{lemma}
 \label{lemm:fahlaoui}
 Let $0 \neq \omega \in H^0(\P^2,\Omega_{\P^2}(2)) \simeq \C^3$ and let $j \in \{1,\ldots,5\}$.
 Consider the pull-back section $\pi^* \omega \in H^0(S,\Omega_S \otimes \pi^* \Oc(2))$.
 If the class of $\omega$ in $\P^2$ is the point $p_j$, then $\pi^* \omega$ vanishes
 at order exactly $2$ along the exceptional divisor $E_j$. Otherwise, it does not vanish along $E_j$.
\end{lemma}
\begin{proof}
Let us assume that the section $\omega$ of
$\Omega_{\P^2}(2)$ is given in homogeneous coordinates $[X:Y:Z]$
such that the blown up point $p_j$ is $[0:0:1]$ by
$$\omega := \frac{XdY-YdX}{Z^2}=\left(\frac{X}{Z}\right)^2 d\left(\frac{Y}{X}\right)\, .$$
In \cite[Exemple 1]{fahlaoui}, it is shown that its pull-back $\pi^*\omega$ on $S$ has poles only along
the strict transform $E_0$ of the line $(Z=0)$ with order two, 
and vanishes with multiplicity two along the exceptional divisor $E_j$ above $p_j$:
$$\pi^*\omega \in H^0(S,\Omega_S\otimes\Oc_S(2E_0-2E_j))\, .$$

Thus the lemma is proved in this case. Since $H^0(\P^2,\Omega_{\P^2}(2)) \simeq \C^3$ and
$\P^2$ is homogeneous under $SL_3$, using the group action, we see that this result holds
whenever the class of $\omega$ corresponds to $p_j$. If this is not the case, then $\omega$
does not vanish at $p_j$, which obviously implies that $\pi^* \omega$ does not vanish along
$E_j$.
\end{proof}

\begin{proposition}
\label{prop:Q3-deg2}
If $S$ is a smooth quadric section of the solid quadric $Q$, then
${\Omega_Q}_{\mid S}$ is stable.
\end{proposition}

The proof runs through the rest of this subsection. We denote by $E_0$ the strict transform of a general line in $\P^2$
so that $\Oc_S(E_0)=\pi^*\Oc_{\P^2}(1)$.
To begin with, consider a line bundle 
$$L=\pi^*\Oc_{\P^2}(-a)\otimes\Oc_S(-\sum_{j=1}^5 b_jE_j)$$
with an inclusion $L\subset{\Omega_Q}_{\mid S}$ which is
seen as a non zero element of $H^0(S,Hom(L, {\Omega_Q}_{\mid S}))$.
\begin{eqnarray*}
\mu({\Omega_Q}_{\mid S})&=&\frac{K_Q\cdot \Oc_{\P^4}(1)_{\mid S}}{3}=\frac{(-5+2)1\times 2\times 2}{3}=-4\\
\mu(L)&=&L\cdot \Oc_{\P^4}(1)_{\mid S}=-3a-\sum b_j\, .
\end{eqnarray*}
\textsl{First step.}
We first intend to show the semi-stability inequality $\mu(L)\leq \mu({\Omega_Q}_{\mid S})$ that is
$$3a+\sum b_j\geq 4,$$
and show that equality can occur only if $L=\Oc_S(-2E_0+2E_j)$.
By the general argument, 
this is ensured if $L$ is the restriction of a line bundle on $\P^4$ i.e.,
a multiple of $\Oc_{\P^4}(1)_{\mid S}$.

The conormal sequence for $S$ in $Q$ reads
\begin{eqnarray}\label{conormal}
0\to \Oc_{\P^4}(-2)_{\mid S}\to {\Omega_Q}_{\mid S}\to \Omega_S\to 0\, .
\end{eqnarray}
If $H^0(Hom(L,\Oc_{\P^4}(-2)_{\mid S}))\neq 0$,
then $\mu(L)\leq \mu (\Oc_{\P^4}(-2)_{\mid S})=-8< \mu({\Omega_Q}_{\mid S})$,
and the desired inequality is proved, in its strict version.

{}From now on, we will assume that $$H^0(Hom(L,\Oc_{\P^4}(-2)_{\mid S}))= 0\, .$$
Hence a non-zero element in $H^0(S,Hom(L, {\Omega_Q}_{\mid S}))$ 
gives a non-zero element in
$$H^0(S,Hom(L, \Omega_S))=H^0(S,\Omega_S\otimes\Oc_{\P^2}(a)\otimes\Oc_S(\sum b_jE_j))\, .$$
In particular, we have an injection
\begin{equation}
\label{equa:injection}
H^0(S,Hom(L, \Omega_S)) \,\hookrightarrow\,
H^0(S \setminus \cup E_j,\Omega_S\otimes\pi^*\Oc_{\P^2}(a))\,=\,
H^0(\P^2 \setminus \cup p_j,\Omega_{\P^2}(a))
\,=\,H^0(\P^2,\Omega_{\P^2}(a))
\end{equation}
and from the Euler sequence with $V:=H^0(\P^2,\Oc_{\P^2}(1))$
$$0\to \Oc_{\P^2}(a-3)\to V^*\otimes\Oc_{\P^2}(a-2)\to T_{\P^2}(a-3)=\Omega_{\P^2}(a)\to 0$$
this leads to $a\geq 2$.

Let $j$ be an integer between $1$ and $5$.
To make use of the sequence~(\ref{eq:cotE}),
we consider a section $\omega_j$ in $H^0(\P^2,\Omega_{\P^2}(2)) \simeq \C^3$ corresponding
to $p_j$, so that, by Lemma \ref{lemm:fahlaoui}, 
$$\pi^*\omega_j\in H^0(S,\Omega_S\otimes\Oc_S(2E_0-2E_j)).$$

If $H^0(Hom(L,\Oc_{S}(-2E_0+2E_j)))\neq 0$, then $\mu(L)\leq \mu (\Oc_{S}(-2E_0+2E_j))=-4=\mu({\Omega_Q}_{\mid S})$
with equality if and only if $L$ is isomorphic to $\Oc_{S}(-2E_0+2E_j)$.

We assume from now on that $H^0(S,Hom(L, \Omega_S))\neq 0$ and that for all $j$,
$H^0(Hom(L,\Oc_{S}(-2E_0+2E_j)))= 0$.
The rational form $\pi^*\omega_j$ yields the sequence
$$0\to\Oc_S(-2E_0+2E_{j})\to\Omega_S\to K_S\otimes\Oc_S(2E_0-2E_j)$$
and after a twist by $L^*$ a map
$$H^0(Hom(L,\Omega_S))\to H^0(L^*\otimes K_S\otimes\Oc_S(2E_0-2E_j))$$
that is injective as $H^0(Hom(L,\Oc_{S}(-2E_0+2E_j)))= 0$, and gives a curve $C_j$ in the linear system
$|L^*\otimes K_S\otimes\Oc_S(2E_0-2E_j)|
=|\pi^*\Oc_{\P^2}(a-1)\otimes\Oc_S((b_j-1)E_j)\otimes\Oc_S(\sum_{k\neq j}(b_k+1)E_k)|$.

The curves $C_j$ are of degree $d=C_j\cdot \Oc_{\P^4}(1)_{\mid S}=3(a-1)+\sum b_j-1+4=\mu(L^*)$.
We hence have to show that $d\geq 4$.
Denote by $C'_j$ the sum of irreducible components of $C_j$ that are not contracted by $\pi$
and by $d'_j\leq d$ its degree. The class of $C'_j$ lies in some linear system
$$|\pi^*\Oc_{\P^2}(a'_j-1)\otimes\Oc_S((b'_j-1)E_j)
\otimes\Oc_S(\sum_{k\neq j}(b'_{jk}+1)E_k) |.$$
As $C'_j-C_j$ consists of effective exceptional curves, $a'_j= a$, $b'_j\leq b_j$ and $b'_{jk}\leq b_k$.

As $E_j$ is a line that is not a component of $C'_j$,
$-b'_j+1=C'_j\cdot E_j\leq d'_j$. Hence,
$b_j\geq b'_j\geq -d'_j+1\geq -d+1$.
The output is
$d=3a+\sum_i b_i\geq 3a-5d+5$ that is 
\begin{eqnarray}\label{eq:d} d\geq \frac{a}{2}+\frac{5}{6}.\end{eqnarray}

Assume $a=2$.
Using the injection (\ref{equa:injection}) and Lemma \ref{lemm:fahlaoui}, a section
in $H^0(S,\Omega_S\otimes\pi^*\Oc_{\P^2}(a))$ can only vanish at order at most $2$ along one exceptional divisor and does not vanish along
the other divisors.
Thus, for all $i$, it holds $b_i \geq -2$,
and there is at most one negative coefficient $b_i$.
This bundle $L=\pi^*\Oc_{\P^2}(-2)\otimes\Oc_S(-\sum_{j=1}^5 b_jE_j)$
is of degree $-6-\Sigma b_i$, which is no more than $-4$, and
we reach the desired inequality.

We now assume $a\geq 3$. It follows from~\eqref{eq:d} that $3\leq d$.
Assume that for some $j$, $C'_j\cdot E_j\geq 3$. Then $d'_j \geq 3$. 
As $S$ is the intersection of two quadrics and as a quadric does not contain any plane curve of degree bigger than $2$,
the curve $C'_j$ cannot be a plane curve. 
Hence, choose a point $x \in C'_j \setminus E_j$ and consider the plane
$\P^2$ generated by the line $E_j$ and $x$.
We get $$d'_j\geq\# \{ \P^2 \cap C'_j \} \geq C'_j\cdot E_j+1\geq 4.$$
 If the linear span $\scal{C'_j}$ of $C'_j$ is a $3$-plane 
then $\scal{C'_j} \cap S$ contains $C'_j \cup E_j$, of degree $d'_j+1\geq 5$,
contradicting $\deg S=4$.
Hence the linear span of $C'_j$ is the whole $\P^4$.
We can choose a line $\ell$ secant to $C'_j$ on two points and disjoint from $E_j$,
and we consider the $3$-plane $\P^3$ generated by $\ell$ and $E_j$. 
We infer $d'_j\geq \# \{ \P^3 \cap C'_j \} \geq C'_j\cdot E_j+2\geq 5$
concluding for stability.

Thus, for all $j$, we have $C'_j\cdot E_j\leq 2$. Therefore $-b'_j+1\leq 2$, $b'_j\geq -1$.
Hence, $d=3a+\sum b_i \geq 3\times 3+5(-1)=~4$, with equality occurring if and only if
$L$ is isomorphic to $\pi^*\Oc_{\P^2}(-3)\otimes\Oc_S(\sum E_i) = \Oc_{\P^4}(-1)_{\mid S}$.
However, we already know that this is not possible.

\textsl{Second step.}
We now prove stability, dealing with the only equality case we encountered.
Namely we will show that $L\,=\,\Oc_S(-2E_0+2E_j)$ is not a subsheaf of $\Omega_Q$.
Let $C=2E_0-\sum E_i$ be the class of the strict transform of the conic in $\P^2$ passing through the five
points $p_i$. Observe that $L=\Oc_{\P^4}(-2) \otimes \Oc_S(2C+2E_j)$, thus a section of $L^* \otimes \Omega_Q$
is a section of $\Omega_Q(2)_{|S}$ that vanishes at order $2$ along $E_j$ and $C$.

The proof of Proposition \ref{prop:Q3-deg2} will therefore be complete once the following lemma is proved:
\begin{lemma}
\label{lemm:section-oq}
Let $s \in H^0(S,\Omega_Q(2)_{|S})$ a non-vanishing section, and let $\Delta_1,\Delta_2$ be two secant lines.
Then $s$ does not vanish at order two along $\Delta_1$ and $\Delta_2$.
\end{lemma}

We will prove this lemma after some preliminary results. First, let us denote by $Q_2$ a quadric cutting out $S$
in $Q$. By simultaneous reduction of quadratic forms, we may assume that the quadric $Q$ is defined by the
identity matrix $I$ and $Q_2$ by some diagonal matrix $D_2$.

Since $H^1(S,\Oc_S)=0$, the section $s$ lifts to a section $\tilde s \in H^0(S,\Omega_{\P^4}(2)_{|S})$. We will
have to consider affine cones: let $U = \C^5 \setminus \{0\}$ and let $p:U \to \P^4$. Whenever
$Z \subset \P^4$ is a subvariety, we denote by $\widehat Z =p^{-1}(Z)$ its affine cone. The section $\tilde s$
defines a section $\wh s \in H^0(\wh S,\Omega_{U|\wh S})$, which can be written as
$\wh s = \sum_{i,j} a_{i,j} Z_i dZ_j$ ($Z_j$ denotes the $j$-th coordinate function on $\C^5$).
We denote by $A$ the matrix $(a_{i,j})$.
Since $\wh s$ is the pull-back of the section $\tilde s$, we have:
\begin{fact}
\label{fact:A-skew}
$A + \tr A$ belongs to the span of $I$ and $D_2$.
\end{fact}

We want to understand the scheme-theoretic vanishing locus of $s$. As a set, it is described~by:
\begin{fact}
\label{fact:s=0}
Let $u \in \wh S$ and $x=[u] \in S$.
Then $s(x)=0$ if and only if $u$ is an eigenvector of~$A$.
\end{fact}
\begin{proof}
The quadratic form $Q$ yields an identification of $\C^5$ with its dual. Moreover, $\wh s(u)$ 
identifies in the basis $dZ_j$ to the column vector $Au$. Since the coordinates have been chosen so that the matrix of $Q$ is
$I$, the tangent space of $\wh Q$ at $u$ has equation $u$ itself. Thus $s(x)$ vanishes if and only if these
two linear forms define the same hyperplane, in other words if and only if $Au$ is a multiple of $u$.
\end{proof}

At first order, the vanishing of $s$ is characterized by:
\begin{fact}
\label{fact:ds=0}
Let $x=[u] \in S$ such that $Au=0$ and $u \not \in \im(A)$.
Let $X=[U] \in T_xS$, with $U \in T_u \wh S$.
Then, the derivative $ds_x(X)$ vanishes if and only if $AU=0$.
\end{fact}
\begin{proof}
As the proof of Fact \ref{fact:s=0} shows, $\Omega_{\wh Q,x}$ identifies with $\C^5/\C \cdot u$.
The statement then follows from the fact that $ds_x(X) = AX \in \C^5/\C \cdot u \simeq \Omega_{\wh Q,x}$.
\end{proof}

We now prove Lemma \ref{lemm:section-oq}. Let $\pi$ be the plane generated by $\Delta_1$ and $\Delta_2$.
Since $s$ vanishes along $\Delta_1$ and $\Delta_2$, by Fact \ref{fact:s=0}, $\wh \pi$ must be included in
an eigenspace of $A$. Replacing $A$ by $A - \lambda \cdot I$ does not change
the section $s$, thus we can assume that $\wh \pi \subset \ker A$. Therefore the rank of $A$ is at most $2$.

Assume first that $A$ has rank $2$. Let $x=[u] \in (\Delta_1 \cup \Delta_2) \setminus \P \im A$. Since $s$
vanishes at order two along $\Delta_1 \cup \Delta_2$, by Fact \ref{fact:ds=0}, we have
$T_u \wh S \subset \ker A$, and so equality of these subspaces. Since we may assume that
$x$ is not the intersection point $\Delta_1 \cup \Delta_2$, we get a contradiction with the following
fact:
\begin{fact}
We have $S \cap \pi = \Delta_1 \cup \Delta_2$. For $x \in \Delta_1 \setminus \Delta_2$,
$\overline{T_xS} \neq \pi$, where $\overline{T_xS} \subset \P^4$ denotes the embedded tangent space.
\end{fact}
\begin{proof}
Let $Q'$ be any quadric containing $S$. We have $Q' \cap \pi = \Delta_1 \cup \Delta_2$ or
$Q' \cap \pi = \pi$, for degree reasons. The first point follows. Assume now that
$x \in \Delta_1 \setminus \Delta_2$ and that $\overline{T_xS} = \pi$. Let $\ell$ be a line through $x$ and
a point $y$ in $\Delta_2 \setminus \Delta_1$. Once again, if $Q'$ is a quadric containing $S$,
then $\ell \cap Q'$ has multiplicity at least $2$ at $x$ 
($\ell\subset\pi=\overline{T_xS}\subset\overline{T_xQ'}$)
and one at $y$, thus $\ell \subset Q'$. This implies that $\ell \subset S$,
contradicting the first point of the Fact.
\end{proof}

Assume now that $A$ has rank $1$. We will use the following observation:
\begin{fact}
\label{fact:matrix-rk-1}
Let $B$ be a square matrix which is the sum of an alternate matrix and a diagonal matrix. Assume that
$\rank B = 1$. Then, up to a permutation of the rows and columns, $B$ can be written as a bloc-diagonal
matrix $\matdd \beta 0 0 0$, with $\beta$ a rank $1$ matrix of order $2$.
\end{fact}
\begin{proof}
Write $B=(b_{i,j})$.
Since a coefficient of $B$ is non zero, a diagonal coefficient of $B$ must be non zero, and assume that
$b_{1,1} \neq 0$. If all the other diagonal coefficients are 0, then we have $b_{i,j}=0$ for
$(i,j) \neq (1,1)$ and the fact is true. In the other case, assume that $b_{2,2} \neq 0$. 
We have $b_{1,3}+b_{3,1}=b_{2,3}+b_{3,2}=b_{1,2}+b_{2,1}=0$ and
$b_{1,1}b_{2,3}-b_{2,1}b_{1,3} = b_{1,1}b_{3,2}-b_{3,1}b_{1,2}=0$, with $b_{1,1},b_{2,2},b_{1,2}$ and
$b_{2,1}$ different from $0$. It follows that
$b_{1,3}=b_{2,3}=b_{3,2}=b_{3,1}=0$. Similarly, all the coefficients $b_{i,j}$ are $0$
except when $i,j \leq 2$.
\end{proof}

Now, $A$ satisfies the hypothesis of Fact \ref{fact:matrix-rk-1}, and moreover the diagonal of $A$ is a linear
combination of $I$ and $D_2$. This implies that a linear combination of $I$ and $D_2$ has rank at most $2$,
contradicting the smoothness of $S$ (in fact, $S$ is smooth if and only if the quadrics in the pencil it defines
all have rank at least $4$). This ends the proof of Lemma \ref{lemm:section-oq}.

To complete the proof of Proposition \ref{prop:Q3-deg2}, one has to consider rank $2$ subsheaves in
$\Omega_Q$. This case follows from the case of rank $1$ subsheaves by the fact that $\Omega_Q$
is self-dual (see the end of the proof of Proposition \ref{prop:Q3-deg1}).

\thispagestyle{empty}

\end{document}